\documentclass[11pt]{amsart}
\usepackage{geometry}                
\usepackage{graphicx}
\usepackage{amssymb}
\usepackage{epstopdf}
\usepackage{enumerate}
\newtheorem{theorem}{Theorem}
\newtheorem{remark}{Remark}
\newtheorem{definition}{Definition}
\newtheorem{lemma}{Lemma}

\newtheorem{corollary}{Corollary}
\newtheorem{assumption}{Assumption}

\newcommand{\be}[1]{\begin{equation} \label{#1}}
\newcommand{\ee}{\end{equation}}

\title[Microlocal analysis of the geodesic X-ray transform]{On the microlocal analysis of the geodesic X-ray transform with conjugate points}
\author{Sean Holman, Gunther Uhlmann}

\begin{document}

\begin{abstract}
We study the microlocal properties of the geodesic X-ray transform $\mathcal{X}$ on a manifold with boundary allowing the presence of conjugate points. Assuming that there are no self-intersecting geodesics and all conjugate pairs are nonsingular we show that the normal operator $\mathcal{N} = \mathcal{X}^t \circ \mathcal{X}$ can be decomposed as the sum of a pseudodifferential operator of order $-1$ and a sum of Fourier integral operators. We also apply this decomposition to prove inversion of $\mathcal{X}$ is only mildly ill-posed in dimension three or higher.
\end{abstract}

\maketitle

\section{Introduction}

Our object of study in this paper will be the geodesic X-ray transform defined on a Riemannian manifold $(M,g)$ with boundary. Loosely speaking, the geodesic X-ray transform, which we denote in the present work by $\mathcal{X}$, is the operator which takes a function defined on $M$ to its integrals along all geodesics of $(M,g)$. The Euclidean version of this problem provides the mathematical basis for X-ray computerised tomography, and has a long history going back at least to the well known work of Radon \cite{Radon17}. Slightly before Radon, Funk also found an inversion formula for the case of symmetric functions on the two-sphere \cite{Funk16}. Study of the general non-Euclidean case began in earnest, to the author's knowledge, with Mukhometov \cite{Mukhometov77} in relation to the boundary rigidity problem which is the problem of determining a Riemannian metric from knowledge of its distance function restricted to the boundary. The boundary rigidity problem arises in seismology when we take the so-called ``travel-time metric" for seismic waves, and consider that we can measure the amount of time it takes for a seismic wave to travel from a source to a receiver location. Indeed, the geodesic X-ray transform of a tensor field arises as a linearization of travel time tomography \cite{Uhlmann01}. Other applications of the geodesic X-ray transform include ultrasound transmission tomography \cite{Opielinski13} and optical tomography with a variable index of refraction \cite{McDowall09}. It has also been shown that injectivity of the geodesic X-ray transform can imply identifiability results for the anisotropic Calder\'on's problem in some special cases \cite{FerreiraKenigSaloUhlmann_weight}.

Most existing results on the geodesic X-ray transform for manifolds with boundary concern simple manifolds. A simple manifold $(M,g)$ is one for which the exponential map centered at any point is a diffeomorphism onto $M$, and $\partial M$ is strictly convex. Using the energy integral method originally introduced by Mukhometov \cite{Mukhometov77}, it can be shown that the geodesic X-ray transform is injective, and stability estimates for the inversion can be obtained, under hypotheses on the sectional curvature which imply that $(M,g)$ is simple (see \cite{Sharafutdinov_book}). It has also been known for some time that the normal operator
\[
\mathcal{N} = \mathcal{X}^t \circ \mathcal{X}
\]
is an elliptic pseudodifferential operator of order $-1$ when $(M,g)$ is a simple manifold \cite{GuilleminSternberg_integral}. The transpose $\mathcal{X}^t$ must be defined by placing an appropriate measure on the space of geodesics. An explicit Fredholm type formula which can be inverted by a Neumann series in the case of a simple manifold was found in the two dimensional case in \cite{PeUh}. The geodesic X-ray transform acting on tensor fields has also been studied extensively almost completely in the case of simple manifolds. See \cite{PaSaUh_survey} and the references therein for a survey of recent progress, and also a listing of current open problems. See \cite{PaSaUh} for the most general existing results on tensor tomography for simple surfaces.

There are few results for non-simple manifolds and most of those that exist are relatively recent. By applying analytic microlocal analysis injectivity of $\mathcal{X}$ can be proven in some non-simple cases \cite{FrigyikStefanovUhlmann_generic,StefanovUhlmann_nonsimple}. One of the most general results to date for dimension three or higher was established using the scattering calculus and a layer stripping argument \cite{UhlmannVasy12}, and applies in some non-simple cases. It is difficult however to connect the hypotheses in \cite{UhlmannVasy12} to other more geometric assumptions on $(M,g)$. Recent work using tools from dynamical systems also shows that $\mathcal{X}$, possibly acting on tensors, is injective for manifolds with hyperbolic trapped sets, although not including any conjugate points \cite{Gu}. Similar results for the transform of a connection are shown in \cite{GuPaSaUh}. In the two dimensional case injectivity of the geodesic X-ray transform was established for some non-simple cases including conjugate points in \cite{Sharafutdinov_layers}, although this injectivity should be held in contrast to other recent work showing that when there are conjugate points present in two dimensions it is not possible to establish a stability estimate for inversion of the geodesic X-ray transform between any Sobolev spaces \cite{MoStUh}. Some results for the tensor problem on a non-simple manifold have also been found in \cite{Da}.

In general, the normal operator $\mathcal{N}$ is not a pseudodifferential operator when $(M,g)$ is not simple. The authors of \cite{StefanovUhlmann12} show that in the case of fold caustics an appropriately localised version of the normal operator is the sum of a pseudodifferential operator and a Fourier integral operator. This result is very much in the same spirit as the results of the current paper, although here we lessen the restriction to fold caustics. In \cite{MoStUh} the restriction to fold caustics is also removed although only in the case of two dimensions. In fact the method of \cite{MoStUh} is similar to the current paper, but here we have analysed the geometry of conjugate points in more detail in order to reach a more general conclusion.

This paper contains two main results. The first is Theorem~\ref{mainthm} which shows that when there are no singular conjugate points the normal operator $\mathcal{N}$ can be decomposed as an elliptic pseudodifferential operator of order $-1$ plus a sum of Fourier integral operators (FIOs). Each FIO corresponds with conjugate pairs of a given order, and the order of the FIO depends on the dimension of $M$ and the order of the conjugate points. The canonical relation of each FIO depends on the geometry of the conjugate pairs. The theorem is actually a bit more general than this, allowing for a weight to be included in $\mathcal{X}$. The second result is Theorem~\ref{contthm} which concerns the use of the decomposition in Theorem~\ref{mainthm} to obtain stability estimates for inversion of $\mathcal{X}$. While results similar to Theorem~\ref{contthm} have been shown before we are not aware of this exact form which we have included because it matches with hypotheses in previous literature on the the convergence of Tikhonov regularisation. This is described in more detail in a remark after the statement of the theorem.

\section{Preliminaries}

In this section we introduce the notation which will be used throughout the paper, a few important definitions, and some preliminary lemmas we will use later. Throughout the paper we make the following assumption on $(M,g)$.

\begin{assumption} \label{basicassume}
$(M, g)$ is an $n$ dimensional compact, Riemannian manifold with smooth strictly convex boundary and with $n \geq 2$. Further assume that $(M,g)$ does not contain any self-intersecting geodesics.
\end{assumption}

\noindent The requirement that $(M,g)$ does not contain self-intersecting geodesics is stronger than the requirement that $(M,g)$ be non-trapping (i.e. not contain any closed geodesics). The prohibition against self-intersecting geodesics arises for technical reasons in the application of the calculus of FIOs, and may be able to be lessened to simply non-trapping via an additional layer of microlocalization, but we have not done this.

The interior of $M$ will be $M^{int} = M \setminus \partial M$. The unit sphere bundle for $M$ (resp. $M^{int}$) will be $S M$ (resp. $S M^{int}$). We will use $\pi$ for the projection mapping $\pi: SM \rightarrow M$, and $i_{SM}$ for the inclusion map $i_{SM}: SM \rightarrow TM$. Further, we will use the same notation for $\pi |_{S M^{int}}$ and $i_{SM} |_{SM^{int}}$ (this should not cause any confusion). We will also have occasion to require other bundle projection mappings, and these will always be denoted by $\pi$ with different subscripts indicating the bundle. For example, for the projection from $T^* S M^{int}$ (this is the cotangent bundle of $S M^{int}$) to $S M^{int}$ we will write $\pi_{T^* S M^{int}}$.

In referring to points of vector bundles a timeless question is whether the notation should include the base point. For example, when indicating points in the tangent bundle $T M$ should one use $(x,v) \in T M$ or just $v \in T M$? We do not fully commit ourselves to either choice, but stipulate that when we write $(x,v) \in T M$ this is actually a short hand for $v \in T M$ and $x = \pi_{T M}(v)$. The same holds for points in other vector bundles. 

At the boundary $\partial M$, the outward pointing unit normal vector will be $\nu_+$, and the set of inward or outward pointing unit vectors will be
\[
\partial_\pm S M = \{ (x,v) \in T M \ : \ x \in \partial M, \ \pm \langle v, \nu_+ \rangle > 0 \}.
\]
Note that $\partial_- S M$ is the set of inward pointing vectors, and $\partial_+ S M$ is the set of outward pointing vectors.

For $v \in T M$ we write $\gamma_v$ for the maximally extended geodesic with initial data $\dot \gamma_v(0) = v$. The reader should specifically note here that we consider $\dot \gamma_v(0)$ to be a point in $T M$, and so, loosely speaking, this initial condition is including both the initial position and tangent vector of the geodesic.

We define two functions $\tau_\pm$ initially on $S M^{int}$ by the formulas
\[
\gamma_{v}(\tau_\pm(v)) \in \partial M,
\]
$\tau_+(v) > 0$, and $\tau_-(v)<0$. These give the backward and forward ``time to the boundary." Since the boundary is strictly convex $\tau_\pm \in C^\infty(S M^{int})$, although $\tau_\pm$ are not smooth up to the boundary \cite{Sharafutdinov_book}. Nonetheless $\tau_\pm$ extend to smooth functions on $S M \setminus T \partial M$, and thus we have $\tau_\pm \in C^\infty(S M \setminus T \partial M)$.

We will make much use of the smooth mappings  $F: S M^{int} \rightarrow \partial_- S M$ defined by
\[
F(v) = \dot \gamma_{v}(\tau_-(v)),
\]
and $\Psi: \mathcal{D} = \{ (v, s) \in S M^{int} \times \mathbb{R} \ : \tau_-(v) < s < \tau_+(v) \} \rightarrow S M^{int}$ defined by
\be{Fdef}
\Psi(v,s) = \dot{\gamma}_v(s).
\ee
The mapping $\Psi$ is the geodesic flow in the unit sphere bundle, and $F$ takes $v \in S M^{int}$ to the unique $w \in \partial_- S M$ such that the tangent vector to the geodesic with initial data $w$ equals $v$ at some point along the geodesic. We will also write $\tilde{\Psi}$ for the geodesic flow on $TM$.

On $M$ we have the Riemannian density which we write as $| \mathrm{d} v_g |$. Further, $g$ induces Liouville densities on $S M$ and $\partial S M$ which we will write as $|\mathrm{d} S M |$ and $|\mathrm{d} \partial S M|$ respectively. With these notations the Santal\'o formula (see \cite{Sharafutdinov_book}) is
\begin{equation}\label{Sant}
\int_{S M} h(v) \ |\mathrm{d} S M (v)| = - \int_{\partial_- S M}  \int_0^{\tau_+(w)} h(\dot \gamma_w(s)) \ \langle w, \nu_+ \rangle \ \mathrm{d} s \ \mathrm | \mathrm{d} \partial S M(w) |
\end{equation}
which holds for $h \in C(SM)$. In view of this formula we use $|\mathrm{d} \mu | = - \langle w, \nu_+ \rangle \ \mathrm | \mathrm{d} \partial S M |$ as the density on $\partial_- S M$.

When dealing with FIOs we need to use half densities, and so we will write $\Omega^{1/2}_M$ for the half density bundle on $M$. The same subscript notation will be used for half densities on other manifolds such as $\Omega^{1/2}_{\partial_- SM}$ for the half densities on $\partial_- SM$. We will also consider function spaces of sections of half-densities writing for example $C_c^\infty(\Omega_{M^{int}}^{1/2})$ for the space of sections of $\Omega_M^{1/2}$ which are smooth with compact support contained inside of $M^{int}$. We now define the main object of our study, the geodesic X-ray transform, acting on half densities.
\begin{definition}
{\bf (Geodesic X-ray transform on half densities)}
The geodesic X-ray transform on $(M, g)$ is defined for $f \in C_c^\infty(\Omega^{1/2}_{M^{int}})$ as a half density on $\partial_- S M$ by the formula
\[
\mathcal{X}[f](v) = \left ( \int_{0}^{\tau_+(v)} \frac{f}{|\mathrm{d} v_g|^{1/2}}(\gamma_v(s)) \ \mathrm{d} s \right ) |\mathrm{d} \mu(v)|^{1/2}.
\]
\end{definition}
\noindent It can be shown that $\mathcal{X}: C_c^\infty(\Omega^{1/2}_{M^{int}}) \rightarrow C_c^\infty(\Omega^{1/2}_{\partial_- S M})$ continuously. In fact $\mathcal{X}$ is an FIO and can be extended to distributions $\mathcal{E}'(\Omega^{1/2}_{M^{int}})$ and various other spaces (including $L^2(\Omega^{1/2}_{M^{int}})$).

Central tools for our analysis will be the push-forward and pull-back of half densities which we now define.
\begin{definition} \label{pushpull_def}
{\bf (Push-forward and pull-back of half densities)} Let $X$ and $Y$ be two manifolds of dimensions $n_X$ and $n_Y$, and let $G: X \rightarrow Y$ be a smooth submersion. Also, suppose we are given half densities $|\mathrm{d} \mu_X|^{1/2}$ and $|\mathrm{d} \mu_Y|^{1/2}$ on $X$ and $Y$. The push-forward $G_*: C_c^\infty(X,\Omega^{1/2}_X) \rightarrow C_c^\infty(\Omega^{1/2}_Y)$ and pull-back $G^*: C_c^\infty(\Omega^{1/2}_Y) \rightarrow C^\infty(\Omega^{1/2}_X)$ are defined by the requirement that
\[
\int_Y h(y)\ G_*[f](y) = \int_X \frac{h}{|\mathrm{d} \mu_Y|^{1/2}}(G(x))\ |\mathrm{d} \mu_X|^{1/2} \ f(x) = \int_X G^*[h](x) \ f(x)
\]
for all $f \in C_c^\infty(\Omega^{1/2}_X)$ and $h \in C_c^\infty(\Omega^{1/2}_Y)$.
\end{definition}
\noindent It can be seen directly from this definition that the pull-back and push-forward are adjoints of one another. Intuitively the pull-back is the precomposition with $G$, while the push-forward is an integration over the level sets of $G$.

Indeed, let us consider the push-forward and pull-back by the map $F: SM^{int} \rightarrow \partial_- SM$ introduced above using the half-densities corresponding to the densities $|\mathrm{d} SM|$ and $|\mathrm{d} \mu|$ also introduced above. Using Santal\'o's formula \eqref{Sant} we find that
\[
\begin{split}
\int_{\partial_- SM} h(w) \ F_*[f](w) & = \int_{SM} \frac{h}{|\mathrm{d} \mu|^{1/2}}(F(v))\ |\mathrm{d} SM(v)|^{1/2}\ f(v) \\
& = \int_{SM} \frac{h}{|\mathrm{d} \mu|^{1/2}}(F(v))\ \frac{f}{|\mathrm{d} SM|^{1/2}}(v)\ |\mathrm{d} SM(v)| \\
& = \int_{\partial_- SM}\left ( \int_0^{\tau_+(w)} \frac{f}{|\mathrm{d} SM|^{1/2}}(\dot \gamma_w(s))\ \mathrm{d} s \right ) \frac{h}{|\mathrm{d} \mu|^{1/2}}(w) \ |\mathrm{d} \mu(w)| \\
& =  \int_{\partial_- SM}\left ( \int_0^{\tau_+(w)} \frac{f}{|\mathrm{d} SM|^{1/2}}(\dot \gamma_w(s))\ \mathrm{d} s \right ) |\mathrm{d} \mu(w)|^{1/2}\ h(w)
\end{split}
\]
From this we have explicit formulae
\be{F_*}
F_*[f](w) = \left ( \int_0^{\tau_+(w)} \frac{f}{|\mathrm{d} SM|^{1/2}}(\dot \gamma_w(s))\ \mathrm{d} s \right ) |\mathrm{d} \mu(w)|^{1/2}
\ee
and
\be{F^*}
F^*[h](v) = \frac{h}{|\mathrm{d} \mu|^{1/2}}(F(v))\ |\mathrm{d} SM(v)|^{1/2}.
\ee
Similar analysis, not requiring the Santal\'o formula, may be applied to $\pi:SM^{int} \rightarrow M$ to find that
\[
\pi_*[f](x) =  \left ( \int_{S_x M} \frac{f}{|\mathrm{d} SM|^{1/2}}(v)\ |\mathrm{d} S_xM(v)| \right ) |\mathrm{d} v_g(x)|^{1/2}
\]
and
\[
\pi^*[h](v) = \frac{h}{|\mathrm{d} v_g|^{1/2}}(\pi(v))\ |\mathrm{d} SM(v)|^{1/2}.
\]
Using these formulae and the observation above that the adjoint of a pull-back is a push-forward, we find that the geodesic X-ray transform and its adjoint may be written as
\[
\mathcal{X} = F_* \circ \pi^* \quad \mbox{and} \quad \mathcal{X}^t = \pi_* \circ F^*.
\]
The normal operator is thus
\[
\mathcal{N} = \pi_* \circ F^* \circ F_* \circ \pi^*.
\]
This will be the jumping off point for our analysis of $\mathcal{N}$ as an FIO.

We actually generalise slightly by adding a weight. Given $\phi \in C^\infty(S M)$ we will write $\phi^m$ for the operator which multiplies by $\phi$, and define
\[
\mathcal{X}_\phi = F_* \circ \phi^m \circ \pi^*
\]
This is the weighted geodesic X-ray transform, and since it requires little extra effort to prove our results for $\mathcal{X}_\phi$ we will do so.

We will now show that it is generally true that the push-forward and pull-back by a submersion are both FIOs. This is certainly a known result going back to at least \cite{GuSc}, but we provide it here for completeness, and so that we have it in precisely the form required for the rest of the current paper.

\begin{lemma} \label{pushpull_lemm}
Suppose we are in the setting of definition \ref{pushpull_def}. Then the pull-back $G_*$ and the push-forward $G^*$ are both FIOs of order $(n_Y-n_X)/4$ with non-vanishing principal symbol. The canonical relation of $G_*$ is
\[
C_{G_*} = \left \{ (\xi, \ DG|_{x}^t\ \xi) \ : \ x \in X, \ \xi \in T^*_{G(x)} Y \setminus \{0\} \right \}
\]
while the canonical relation of $C_{G^*}$ is 
\[
C_{G^*} = \left \{ (DG|_{x}^t\ \xi, \xi) \ : \ x \in X, \ \xi \in T^*_{G(x)} Y \setminus \{0\} \right \}.
\]
\end{lemma}

\begin{proof}
Let $K_{G_*}$ be the Schwarz kernel of $G_*$. From definition \ref{pushpull_def}, for any $f \in C_c^\infty(\Omega^{1/2}_X)$ and $h \in C_c^\infty( \Omega^{1/2}_Y)$ we have
\[
\langle K_{\mathcal{X}_{S M}}, h \otimes f \rangle = \int_X \frac{h}{|\mathrm{d} \mu_Y|^{1/2}}(G(x))\ \frac{f}{|\mathrm{d} \mu_X|^{1/2}}(x) \ |\mathrm{d} \mu_X(x)| 
\]
Suppose now that the support of $f$ and the support of $h$ are contained in the domains of single coordinate charts $\phi: X \rightarrow \mathbb{R}^{n_X}$ and $\tilde{\phi}: Y \rightarrow \mathbb{R}^{n_Y}$ respectively, that
\[
\tilde{G} = \tilde{\phi} \circ G \circ \phi^{-1}: \mathbb{R}^{n_X} \rightarrow \mathbb{R}^{n_Y}
\]
is the representation of $G$ in these coordinates, that
\[
J(x)\ |\mathrm{d} x|  = \phi_* (| \mathrm{d} \mu_X | )
\]
is the representation of the relevant density in the coordinates on $X$, and that
\[
\tilde{f}(x) = \frac{f}{|\mathrm{d} \mu_X|^{1/2}}( \phi^{-1}(x)), \quad \tilde{h}(y) = \frac{h}{|\mathrm{d} \mu_Y|^{1/2}}(\tilde{\phi}^{-1}(y)).
\]
Then using the Fourier inversion formula
\be{XSOkernel}
\begin{split}
\langle K_{\mathcal{X}_{S M}}, h \otimes f \rangle &= \int_{\mathbb{R}^{n_X}} \tilde{h}(\tilde{G}(x)) \ \tilde{f}(x) \ J(x) \ \mathrm{d} x\\
& = \frac{1}{(2 \pi)^{n_Y}}\int_{\mathbb{R}_y^{n_Y} \times \mathbb{R}_\xi^{n_Y} \times \mathbb{R}_x^{n_X}} e^{i \xi \cdot (\tilde{G}(x)-y)} \tilde{h}(y) \ \tilde{f}(x) \ J(x)\ \mathrm{d} y \ \mathrm{d} \xi\ \mathrm{d} x.
\end{split}
\ee
Untwining the definitions we can see that this is a local representation of the Schwartz kernel of an FIO of order $(n_Y-n_X)/4$. From the phase of the local representation \eqref{XSOkernel} we can further see that the claimed canonical relations are correct. Note that since $DG|_{v}$ is a surjective map at every $v$, $DG|_v^t$ is injective and so in fact $C_{G^*}$ and $C_{G_*}$ are canonical relations. Finally, from the local representation \eqref{XSOkernel} we see that the principal symbol is equal to $J(x)$ multiplied by an appropriate non-vanishing half density, and so the principal symbol also does not vanish.
\end{proof}
\noindent Since the dimensions of $M$, $\partial_- SM$ and $SM$ are $n$, $2n-2$, and $2n-1$ respectively, Lemma \ref{pushpull_lemm} shows in particular that
\[
F_* \in \mathcal{I}^{-1/4}(\partial_- SM \times SM,\ C'_{F_*}; \ \Omega^{1/2}_{\partial_- SM \times SM}), \ F^* \in \mathcal{I}^{-1/4}(SM \times \partial_- SM,\ C'_{F^*}; \ \Omega^{1/2}_{\partial_- SM \times SM}),
\]
\[
\pi_* \in \mathcal{I}^{(1-n)/4}(M \times SM, \ C'_{\pi_*}; \ \Omega^{1/2}_{M \times SM} ), \  \mbox{and} \ \pi^* \in \mathcal{I}^{(1-n)/4} (SM \times M, \ C'_{\pi^*}; \ \Omega^{1/2}_{SM \times M}).
\]
The bulk of the remaining analysis on the normal operator $\mathcal{N}$ will be to show how the clean composition calculus may be applied to analyse the compositions of these operators. First however, we review some more preliminary material.

The canonical symplectic form on $T^* M$ will be denoted $\omega$, and the corresponding one on $T M$ induced by $g$ denoted by $\omega_g$. Specifically, if $\flat_g: T M^{int} \rightarrow T^* M^{int}$ is the musical isomorphism given by $g$ then $\omega_g$ is defined by
\[
\omega_g(X, Y) = \omega(D \flat_g X, D \flat_g Y)
\]
for any $X$ and $Y$ in $T(T M^{int})$. This symplectic form then induces an isomorphism $\flat_{\omega_g}$ from $T(T M^{int})$ to $T^*(T M^{int})$ at each point defined by
\[
[\flat_{\omega_g}(X)](Y) = \omega_g(X,Y) := \mathfrak{i}_X \omega_g(Y)
\]
for all $X$ and $Y$ in $T(T M^{int})$ (here $\mathfrak{i}_X$ is interior multiplication by $X$ which is defined by the previous formula). We denote the inverse of this map in the usual way by $\sharp_{\omega_g}$.

The geodesic flow on $T M$ is given by the Hamiltonian $H(v) =1/2 (\| v \|_g^2 - 1) \in C^\infty(T M)$. Thus we have for $v \in T M^{int}$ the following invariant formula for $\ddot \gamma_v(0) \in T (T M^{int})$
\[
\ddot \gamma_v(0) = X^{\omega_g}_H(v)
\]
where $X^{\omega_g}_H(v)$ is the Hamiltonian vector field given by $H$ with respect to the symplectic form $\omega_g$ evaluated at $v$. Since $S M$ is a level surface for $H$ in fact $X^{\omega_g}_H(v) \in T S M^{int}$ and we see that $\ddot \gamma_v(0)$ defines a smooth vector field on $S M^{int}$. We record these observations as a lemma which will be useful later.

\begin{lemma} \label{Xsmooth}
The smooth vector field $X^{\omega_g}_H$ on $TM$ restricts to a smooth vector field on $S M^{int}$ and for each $v \in S M^{int}$
\[
\ddot \gamma_v(0) = X^{\omega_g}_H(v).
\]
\end{lemma}

\section{Characterisation of $\mathcal{N}$}

In this section we reach the main result of the paper, which is a characterisation of $\mathcal{N}_\phi$ as a sum of Fourier integral operators. To do this we use the decompostion described in the previous section
\be{Ndef}
\mathcal{N}_\phi = \pi_* \circ \phi^m\circ  F^* \circ F_* \circ \phi^m \circ \pi^*.
\ee
Our first result in this direction is the following theorem.
\begin{theorem} \label{NSOthm}
Define
\[
\mathcal{N}_{S M} = F^* \circ F_*.
\]
Then $\mathcal{N}_{S M} \in \mathcal{I}^{-1/2}\left (S M^{int} \times S M^{int},\ C'_{\mathcal{N}_{S M}};\ \Omega^{1/2}_{S M^{int} \times S M^{int}} \right )$ where
\[
C_{\mathcal{N}_{S M}} = \left \{ \left ( DF|^t_{v} \xi, DF|^t_{\tilde{v}} \xi \right ) \ : \ v,\ \tilde{v} \in S M^{int}, \ F(v) = F(\tilde{v}), \ \xi \in T^*_{F(v)} \partial_- S M \right \}.
\]
Furthermore, the principal symbol of $\mathcal{N}_{SM}$ does not vanish.
\end{theorem}

\begin{remark}
It is worthwhile to note that, while we are using the abstract FIO composition calculus to analyze the operators in this theorem, in fact it is not hard at all to find an explicit formula for $\mathcal{N}_{S M}$. Indeed, using \eqref{F_*} and \eqref{F^*} we can see that
\[
\mathcal{N}_{S M}[h](\nu) = \int_{\tau_-(\nu)}^{\tau_+(\nu)} \frac{h}{|\mathrm{d} SM|^{1/2}} (\dot \gamma_{\nu}(s))\ \mathrm{d} s\ |\mathrm{d} SM|^{1/2}.
\]
\end{remark}
\begin{proof}
The proof will be an application of the clean composition calculus for FIOs. There is an adjustment required since $F^*$ is not properly supported. However, using the fact that
\[
F_*: C_c^\infty( \Omega^{1/2}_{SM^{int}}) \rightarrow C_c^\infty( \Omega^{1/2}_{\partial_- SM^{int}})
\]
continuously and $F^*: C_c^\infty( \Omega^{1/2}_{\partial_- SM^{int}}) \rightarrow C^\infty( \Omega^{1/2}_{SM^{int}})$ there is no issue defining the composition of the two operators, and we may localize and reduce to the case of a composition of two properly supported FIOs. Therefore the proof is reduced to analysis of the canonical relations.

Using Lemma~\ref{pushpull_lemm}, it is straightforward to see that
\[
C_{\mathcal{N}_{S M}} = C_{F^*} \circ C_{F_*},
\]
where $C_{\mathcal{N}_{S M}}$ is defined in the statement of Theorem~\ref{NSOthm}, and so if we can show that this composition is clean the proof will be almost complete. To show the composition is clean we must show:
\begin{enumerate}[(i)]
\item The intersection
\begin{equation}\label{C}
C = (C_{F^*} \times C_{F_*}) \cap (T^* S M^{int} \times \Delta(T^* \partial_- S M) \times T^* S M^{int})
\end{equation}
is clean in the sense that $C$ is an embedded submanifold, and at every point $c \in C$ the tangent space $T_c C$ is equal to the intersection of the tangent spaces of the two manifolds being intersected.
\item The projection map $\pi_C : C \rightarrow T^* S M^{int} \times T^* S M^{int}$ is proper.
\item For every $\phi \in T^* S M^{int} \times T^* S M^{int}$, $\pi_C^{-1}(\phi)$ is connected.
\end{enumerate}
To begin we note that since $DF|_v$ is surjective for every $v$, $DF|_v^t$ is injective for every $v$. This implies that $\pi_C$ is injective, and so point (iii) is immediate and assuming the other points are proven the excess is zero. Since the principal symbols of each of the two operators is not vanishing this will also show that the principal symbol of the composition is not vanishing, and so in fact all that remains is to demonstrate points (i) and (ii) hold.

To begin proving point (i) we first observe that $\ddot \gamma_v(0)$ is in the kernel of $DF|_v$ for all $v \in S M^{int}$, and since $DF|_v$ is surjective by a dimension count the full kernel of $DF|_v$ is the span of $\ddot \gamma_v(0)$. By Lemma~\ref{Xsmooth} the range of $DF|_v^t$ is thus
\[
\left \{ \theta \in T^*_v S M^{int} \ : \ \theta \left ( X^{\omega_g}_H(v) \right ) = 0 \right \}.
\]
Letting $v$ vary, and using the nontrapping assumption on $(M,g)$, we find that the range of the map
\[
G: \{ (\xi,s) \in T^* \partial_- S M \times \mathbb{R} \ : \ 0 < s < \tau_+(\xi) \} \ni (\xi,s) \mapsto DF|_{\dot \gamma_{\pi_{T^* S M}(\xi)}(s)}^t \xi \in T^* S M^{int}
\]
is precisely the set
\begin{equation} \label{image}
\left \{ \theta \in T^* S M^{int} \ : \ \theta \left ( X^{\omega_g}_H(\pi_{T^* S M}(\theta)) \right ) = 0 \right \}.
\end{equation}
Again by Lemma~\ref{Xsmooth} the map
\[
T^* S M^{int}\ni \theta \mapsto \theta \left ( X^{\omega_g}_H(\pi_{T^* S M}(\theta)) \right )
\]
is a smooth submersion and so the set \eqref{image}, as the zero level set of a smooth submersion, is a closed embedded submanifold of $T^* S M^{int}$. Further, the map $G$ is a diffeomorphism onto that set as can be shown using the inverse function theorem. We will write the inverse of $G$ composed with projections onto the $\xi$ and $s$ components as $G^{-1}_\xi$ and $G^{-1}_s$ respectively

Now note that $C$ can be parametrized by a mapping
\[
P_C: \mathcal{D}_{P_C} = \{(\xi, s , \tilde{s}) \in T^* \partial_- S M \times \mathbb{R} \times \mathbb{R} \ : \ 0 < s, \tilde{s} < \tau_+(\pi_{T^* S M}(\xi))\} \rightarrow C
\]
defined by
\[
P_C(\xi,s, \tilde{s}) = \left ( DF|_{\dot \gamma_{\pi_{T^* S M}(\xi)}(s)}^t \xi, \xi, \xi, DF|_{\dot \gamma_{\pi_{T^* S M}(\xi)}(\tilde{s})}^t \xi \right ).
\]
We initially have directly from its construction that $P_C$ is an injective immersion into $\mathrm{Range}(G) \times T^* \partial_- S M \times T^* \partial_- S M \times \mathrm{Range}(G)$, and to show that it is a smooth embedding it only remains to show that the inverse is continuous. To see the inverse $P_C^{-1}$ is continuous we note that it is the restriction to $C$ of the map
\[
P^{ext,-1}_C: \mathrm{Range}(G) \times T^* \partial_- S M \times T^* \partial_- S M \times \mathrm{Range}(G) \rightarrow \mathcal{D}_{P_C}
\]
defined by
\[
P^{ext,-1}_C(A_1, A_2, A_3, A_4) = \Big (G^{-1}_\xi(A_1), G^{-1}_s(A_1), G^{-1}_s(A_4)\Big ).
\]
Since $G$ is a diffeomorphism, in particular $P^{ext,-1}_C$ is continuous, and so is $P_C^{-1}$. Thus $P_C$ is a smooth embedding and $C$ is a $4n - 2$ dimensional embedded submanifold.

To check that the intersection is clean note that, similar to $C$, $C_{F^*} \times C_{F_*}$ can be parametrized by a related mapping
\[
\begin{split}
P_{C_{F^*} \times C_{F_*}} : \mathcal{D}_{P_{C_{F^*} \times C_{F_*}}}  = \{(&\xi,\tilde{\xi},s, \tilde{s}) \in T^* \partial_- S M \times T^* \partial_- S M\times \mathbb{R} \times \mathbb{R} \ : \\
& \qquad  0 < s < \tau_+(\pi_{T^* S M}(\xi)), \quad  0 < \tilde{s} < \tau_+ (\pi_{T^* S M}(\tilde{\xi}))\} \\
& \rightarrow T^* S M^{int} \times T^* \partial_- S M \times T^* \partial_- S M \times T^* S M^{int}
\end{split}
\]
defined by
\[
P_{C_{F^*} \times C_{F_*}}(\xi,\tilde{\xi},s, \tilde{s}) = \left ( DF|_{\dot \gamma_{\pi_{T^* S M}(\xi)}(s)}^t \xi, \xi, \tilde{\xi}, DF|_{\dot \gamma_{\pi_{T^* S M}(\tilde{\xi})}(\tilde{s})}^t \tilde{\xi} \right ).
\]
Suppose that $c_0 \in C$. Then any tangent vector $X \in T_{c_0} (C_{F^*} \times C_{F_*})$ comes from a smooth curve
\[
\mathbb{R} \ni \alpha \mapsto \left ( \xi(\alpha), \tilde{\xi}(\alpha), s(\alpha), \tilde{s}(\alpha) \right )
\]
such that $P_{C_{F^*} \times C_{F_*}}\left ( \xi(0), \tilde{\xi}(0), s(0), \tilde{s}(0) \right ) = c_0$ via
\[
X = \left . \frac{\mathrm{d}}{\mathrm{d} \alpha} \right |_{\alpha=0} P_{C_{F^*} \times C_{F_*}} \left ( \xi(\alpha), \tilde{\xi}(\alpha), s(\alpha), \tilde{s}(\alpha) \right ).
\]
This tangent vector $X$ is also in $T_{c_0}(T^* S M^{int} \times \Delta(T^* \partial_- S M) \times T^* S M^{int})$ if and only if
\[
\left . \frac{\mathrm{d}}{\mathrm{d} \alpha} \right |_{\alpha=0} \xi(\alpha) = \left . \frac{\mathrm{d}}{\mathrm{d} \alpha} \right |_{\alpha=0} \tilde{\xi}(\alpha).
\]
Thus expressed in any local coordinate system $\xi$ and $\tilde{\xi}$ agree to first order at $\alpha = 0$, and so in fact
\[
\begin{split}
X & = \left . \frac{\mathrm{d}}{\mathrm{d} \alpha} \right |_{\alpha=0} P_{C_{F^*} \times C_{F_*}} \left ( \xi(\alpha), \tilde{\xi}(\alpha), s(\alpha), \tilde{s}(\alpha) \right )\\
& =  \left . \frac{\mathrm{d}}{\mathrm{d} \alpha} \right |_{\alpha=0} P_{C_{F^*} \times C_{F_*}} \left ( \xi(\alpha), \xi(\alpha), s(\alpha), \tilde{s}(\alpha) \right )\\
& = \left . \frac{\mathrm{d}}{\mathrm{d} \alpha} \right |_{\alpha=0} P_C \left ( \xi(\alpha), s(\alpha), \tilde{s}(\alpha) \right ).
\end{split}
\]
Therefore $X \in T_{c_0} C$, and the intersection is clean. Thus point (i) is proven and all that remains is point (ii).

For point (ii) we note that if $K \Subset T^* S M^{int} \times T^* S M^{int}$ is a compact set, then for some constant $\epsilon >0$ we have that
\[
\begin{split}
K \Subset L := \{ (\theta,\tilde{\theta})\in T^* S M^{int} \times T^* S M^{int} \ : & \ \mathrm{dist}(\pi \circ \pi_{T^*S M^{int}}(\theta),\partial M) \geq \epsilon, \\
& \mathrm{dist}(\pi \circ \pi_{T^* S M^{int}}(\tilde{\theta}),\partial M) \geq \epsilon,  \\
& (\|\theta\|_g + \| \tilde{\theta}\|_g ) \leq \epsilon^{-1} \}
\end{split}
\]
where $L$ is compact. Now, because the boundary of $M$ is strictly convex there exists $\tilde{\epsilon}>0$ such that for $(\theta,\tilde{\theta}) \in L$
\[
\langle F(\pi_{T^*S M^{int}}(\theta)), \nu_+ \rangle_g < -\tilde{\epsilon}, \quad \langle F(\pi_{T^*S M^{int}}(\tilde{\theta})), \nu_+\rangle_g < -\tilde{\epsilon},
\]
\[
\quad \tau_+(\pi_{T^*S M^{int}}(\theta)),\ \tau_+(\pi_{T^*S M^{int}}(\tilde{\theta})) \geq \epsilon, \quad \quad \tau_-(\pi_{T^*S M^{int}}(\theta)),\ \tau_-(\pi_{T^*S M^{int}}(\tilde{\theta})) \leq -\epsilon.
\]
Therefore there exists $\tilde{M}$ such that
\[
\begin{split}
P_C^{-1}(\pi_C^{-1}(L)) \subset \big \{ (\xi,s, \tilde{s}) \in \mathcal{D}_{P_C} \ : & \ \langle \pi_{T^* \partial_- S M}(\xi), \nu_+ \rangle_g < -\tilde{\epsilon}, \\
& \epsilon \leq s, \ \tilde{s} \leq \tau_+(\pi_{T^* \partial_-S M}(\xi)) - \epsilon,\\
&  \big \|\xi\|_g \leq \tilde{M} \}.
\end{split}
\]
This last set is compact and since $P_C^{-1}(\pi_C^{-1}(K)) \subset P_C^{-1}(\pi_C^{-1}(L))$ we have that $P_C^{-1}(\pi_C^{-1}(K))$ is a closed subset of a compact set, and therefore is compact itself. Since $P_C$ is a diffeomorphism this implies that $\pi_C$ is a proper map and thus completes the proof.
\end{proof}

Now let $\chi \in C^\infty(S M^{int} \times S M^{int})$. We will eventually need to cut up the operator $\mathcal{N}_{SM}$ using such functions. Indeed, by the notation $\chi \mathcal{N}_{SM}$ we mean the operator whose Schwartz kernel is given by $\chi K_{\mathcal{N}_{SM}}$ where $K_{\mathcal{N}_{SM}}$ is the Schwartz kernel of $\mathcal{N}_{SM}$. With this in mind, for the next step we consider compositions of the form
\[
(\chi \mathcal{N}_{S M}) \circ \pi^*.
\]
As we will see in the next theorem this is still a Fourier integral operator.

\begin{theorem}
Define
\[
L_\chi = (\chi N_{S M}) \circ \pi^*.
\]
Then $L_{\chi} \in \mathcal{I}^{-(n+1)/4}(SM^{int} \times M^{int},\ C'_{L};\ \Omega^{1/2}_{SM^{int} \times M^{int}})$ where
\be{LambdaL}
\begin{split}
C_{L} = \{ (DF|_{v}^t \xi, \tilde{\eta}) \in T^* SM^{int} \times T^* M^{int} \ & : \ v 
\in S M^{int}, \ \exists\ \tilde{v} \in S M^{int} \\
&\hskip-1in \mbox{such that} \ \pi(\tilde{v}) = \pi_{T^* M}(\tilde{\eta}), \ F(v) = F(\tilde{v}), \ DF|_{\tilde{v}}^t \xi = D\pi |_{\tilde{v}}^t \tilde{\eta} \}.
\end{split}
\ee
Furthermore, the principal symbol of $L_\chi$ only vanishes at points $(DF|_{v}^t \xi, \tilde{\eta})$ for which, when $\tilde{v}$ satisfies the requirement in \eqref{LambdaL}, $\chi(v,\tilde{v}) = 0$.
\end{theorem}

\begin{remark}
As with $N_{SM}$ we can easily find an explicit form for the action of $L_\chi$ which is
\[
L_\chi[f](\nu) = \int_{\tau_-(\nu)}^{\tau_+(\nu)} \chi(\nu, \dot{\gamma}_\nu(s)) \frac{f}{|\mathrm{d} SM|^{1/2}} (\gamma_{\nu}(s)) \ \mathrm{d} s\ |\mathrm{d} SM|^{1/2}.
\]
\end{remark}

\begin{proof}
Composing the canonical relation for $\pi^*$ given by Lemma \ref{pushpull_lemm} with that for $\mathcal{N}_{SM}$ given in Theorem~\ref{NSOthm} we obtain $C_{L_\chi}$ given in \eqref{LambdaL}. If, as in the proof of Theorem~\ref{NSOthm}, we show that the composition of the canonical relations is clean then this will complete the proof as before.

We refer to points (i)-(iii) given in the proof of Theorem~\ref{NSOthm}, and again show that they hold in the present case in which equation \eqref{C} is replaced by
\[
C = (C_{\mathcal{N}_{S M}} \times C_{\pi^*}) \cap (T^* S M^{int} \times \Delta(T^* S M) \times T^* M^{int}).
\]
First, similar to the previous proof, since $D \pi |^t_v$ is injective for all $v$, and by the hypothesis that there are no self-intersecting geodesics, $\pi_C$ is injective and so (iii) holds and the excess is $e = 0$. The statement concerning vanishing of the principal symbol will follow from the clean composition calculus as well, and so all that remains is to establish (i) and (ii).

To begin proving the other parts we introduce a sub-bundle $V$ of $T^* S M$ with fibres over each point $v \in S M$ defined by
\[
V_v = \mathrm{Range} (D\pi|^t_v).
\]
For every $v$, $D \pi|_v^t$ is a linear isomorphism from $T^*_{\pi(v)} M$ to $V_v$, and we denote the inverse by $(D \pi|_v^t)^{-1}: V_v \rightarrow T^*_{\pi(v)} M$. When these maps for each fibre are combined together the result is a bundle mapping which we will write as $D \pi^{-t}: V \rightarrow T^* M$. In order to parametrize $C$ we introduce an embedded submanifold $\mathcal{O}$ of $V$ defined by
\be{mcO}
\mathcal{O} = \Big \{ \zeta \in V  \ : \ D \pi^{-t}\zeta \big ( \pi_{V}(\zeta) \big ) = 0 \Big \}.
\ee
Using the same mapping $G$ as in the proof of Theorem~\ref{NSOthm}, let us now show that $\mathcal{O} = \mathrm{Range}(G) \cap V$. Indeed, if $\zeta \in V$, then
\[
\begin{split}
\zeta \big (X_H^{\omega_g}(\pi_V(\zeta))\big ) & = D \pi^{-t} \zeta \big ( D \pi X_H^{\omega_g}(\pi_V(\zeta)) \big )\\
& = D \pi^{-t} \zeta \big ( \dot \gamma_{\pi_V(\zeta)}(0) \big ) \\
& = D \pi^{-t} \zeta \big ( \pi_V(\zeta) \big ).
\end{split}
\]
Since $\zeta \in \mathrm{Range}(G)$ if and only if $\zeta \big (X_H^{\omega_g}(\pi_V(\zeta))\big ) = 0$, this proves the claim that $\mathcal{O} = \mathrm{Range}(G) \cap V$.

Now we use the set $\mathcal{O}$ to parametrize $C$ via the map
\[
P_C: \mathcal{D}_{P_C} = \Big \{ \big ( \zeta , s \big ) \in \mathcal{O} \times \mathbb{R}\ : \ \tau_{-}(\pi_V(\zeta)) < s < \tau_+(\pi_{V}(\zeta)) \Big \} \rightarrow C
\]
defined by
\[
P_C\big ( \zeta , s \big ) = \Big (D \Psi(\cdot,s)|_{\Psi(\pi_V(\zeta),s)}^t \zeta , \zeta, \zeta, D \pi^{-t} \zeta \Big )
\]
where we recall that $\Psi$ is the geodesic flow in $S M$. Note that since for any $v \in S M$, and any $s$ where the right hand side is defined,
\[
F(v) = F (\Psi(v,s)),
\]
we have that
\begin{equation}\label{Psisemi}
DF|_v^t = D \Psi(\cdot,s)|_{v}^t D F|_{\Psi(v,s)}^t,
\end{equation}
and so $D \Psi(\cdot,s)|_v^t$ preserves $\mathrm{Range}(G)$. Thus we see that $P_C$ is an injective immersion into $\mathrm{Range}(G) \times \mathrm{Range}(G) \times \mathrm{Range}(G) \times T^*M$. To show $P_C^{-1}$ is continuous we note that it is the restriction to $C$ of the map
\[
P_C^{ext,-1}:\mathrm{Range}(G) \times \mathrm{Range}(G) \times \mathrm{Range}(G) \times T^*M \rightarrow \mathcal{D}_{P_C}
\]
defined by
\[
P_C^{ext,-1}(A,B,C,D) = \big ( B, G^{-1}_s(B) - G^{-1}_s(A) \big ).
\]
Thus $P_C^{-1}$ is continuous, and so $P_C$ is a smooth embedding and $C$ is a $3n -1$ dimensional embedded submanifold.

To complete the proof of (i) in this case note that $C_{\mathcal{N}_{S M}} \times C_{\pi^*}$ can be parametrized by the mapping
\[
\begin{split}
P_{C_{\mathcal{N}_{S M}} \times C_{\pi^*}} : \mathcal{D}_{P_{C_{\mathcal{N}_{S M}} \times C_{\pi^*}}} = \{ & (\xi, \zeta, s,\tilde{s})\in T^* \partial_- S M \times V \times \mathbb{R}\times \mathbb{R}\ : \\
& \qquad 0 < \tilde{s} < \tau_+(\xi), \ \tilde{s} - \tau_+(\xi) < s <  \tilde{s} \} \\
&\qquad  \rightarrow T^* S M^{int} \times T^* S M^{int} \times T^* S M^{int} \times T^* M^{int}
\end{split}
\]
defined by
\[
\begin{split}
P_{C_{\mathcal{N}_{S M}} \times C_{\pi^*}}(\xi,\tilde{\eta},s,\tilde{s}) & = \Big ( DF|_{\dot \gamma_{\pi_{T^* S M}(\xi)}(\tilde{s} - s)}^t \xi, DF|_{\dot \gamma_{\pi_{T^* S M}(\xi)}(\tilde{s})}^t \xi, \zeta, D \pi^{-t} \zeta \Big ) \\
& =  \Big ( D \Psi(\cdot,s)|_{\dot \gamma_{\pi_{T^* S M}(\xi)}(\tilde{s}-s)}^t DF|_{\dot \gamma_{\pi_{T^* S M}(\xi)}(\tilde{s})}^t \xi, DF|_{\dot \gamma_{\pi_{T^* S M}(\xi)}(\tilde{s})}^t \xi, \zeta, D \pi^{-t} \zeta \Big )
\end{split}
\]
We now follow precisely the same reasoning as in the proof of Theorem~\ref{NSOthm}. Indeed, if $c_0 \in C$, and $X \in T_{c_0}( C_{\mathcal{N}_{S M}} \times C_{\pi^*})$, then there is a smooth curve $\alpha \mapsto (\xi(\alpha),\zeta(\alpha), s(\alpha),\tilde{s}(\alpha))$ such that $P_{C_{\mathcal{N}_{S M}} \times C_{\pi^*}}(\xi(0),\zeta(0), s(0),\tilde{s}(0)) = c_0$ and the derivative along this curve at zero is $X$. As in the previous proof, if $X \in T_{c_0}(T^* S M^{int} \times \Delta(T^* S M^{int}) \times T^* M^{int})$ as well then in any local coordinates
\[
\zeta'(0) = \left . \frac{\mathrm{d}}{\mathrm{d} \alpha}\right |_{\alpha = 0} DF|_{\dot \gamma_{\pi_{T^* S M}(\xi(\alpha))}(\tilde{s}(\alpha))}^t \xi(\alpha).
\]
From this, noting also that
\[
\dot \gamma_{\pi_{T^* S M}(\xi)}(\tilde{s}-s) = \Psi\left (\pi_V \left (DF|_{\dot \gamma_{\pi_{T^* S M}(\xi)}(\tilde{s})}^t \xi \right), -s \right ),
\]
we conclude that
\[
X = \left .\frac{\mathrm{d}}{\mathrm{d} \alpha}\right |_{\alpha = 0} P_{C_{\mathcal{N}_{S M}} \times C_{\pi^*}}(\xi(\alpha),\zeta(\alpha), s(\alpha),\tilde{s}(\alpha)) = \left .\frac{\mathrm{d}}{\mathrm{d} \alpha}\right |_{\alpha = 0} P_C(\zeta(\alpha), s(\alpha))
\]
and so $X \in T_{c_0}C$. This completes the proof of (i) in this case.

The proof of (ii) follows exactly the same procedure as in Theorem~\ref{NSOthm}, and we will not repeat the details here.
\end{proof}

We finally wish to complete the analysis of $\mathcal{N}_\phi$ by considering the composition
\be{N2comp}
\mathcal{N}_\phi = \pi_* \circ L_{\phi \otimes \phi}.
\ee
We comment that, following from the previous remarks, it is not difficult to show that $\mathcal{N}_\phi$ is given explicitly by the formula
\be{explicitN}
\mathcal{N}_\phi[f](x) = \int_{S_x M} \phi(\nu) \left ( \int_{\tau_-(\nu)}^{\tau_+(\nu)} \frac{f}{|\mathrm{d} v_g|^{1/2}}(\gamma_{\nu}(s)) \ \phi(\dot\gamma_\nu(s))\ \mathrm{d} s \right ) \ |\mathrm{d} S_xM(\nu)|^{1/2}.
\ee
The composition of the canonical relations in this case is not as simple as it was in the previous theorems. Indeed, the composition in general has multiple connected components with one component giving a pseudodifferential operator, and the others occurring only when there are conjugate points. Under some additional hypotheses the extra components which appear in the case of conjugate points give rise to Fourier integral operators, whose canonical relations and orders can be calculated. This is shown below in Theorem~\ref{mainthm} which is our main result. For the statement of Theorem~\ref{mainthm} we use the following definition, notations, and lemmas which also clarify the geometric structure of the operators involved. 

\begin{definition} {\bf (Conjugate pairs)} \label{Conjugate pairs}
A pair of unit tangent vectors $(v ,\tilde{v}) \in S M \times S M$ is called a conjugate pair if $v$ and $\tilde{v}$ are both tangent to one geodesic, and
\[
K_{(v,\tilde{v})} = \mathrm{ker} \Big ( D \pi|_{\tilde{v}} \circ D_v \Psi |_{(v,\tau_+(v) - \tau_+(\tilde{v}))} \Big ) \bigcap \mathrm{ker} \Big ( D \pi |_{v} \Big ) \neq \{0\}.
\]
If the dimension of $K_{(v,\tilde{v})}$ is $k \geq 1$, then $(v,\tilde{v})$ is a conjugate pair of order $k$. We define the set $C_R$ of pairs of regular conjugate vectors to be the subset of $S M^{int} \times S M^{int}$ such that given any $(v,\tilde{v}) \in C_R$ there is a neighborhood $U$ of $(v,\tilde{v})$ in $S M^{int} \times S M^{int}$ such that any other pair $(w,\tilde{w}) \in U$ which is also conjugate is conjugate of the same order. Further, $C_{R,k}$ will be the subset of regular pairs of conjugate vectors having order $k$. Finally, the set of pairs of conjugate vectors $C_S$ which are not regular will be called the set of singular conjugate vectors.
\end{definition}

\noindent Of course it is standard to define conjugate points along geodesics in terms of vanishing Jacobi fields, and in fact Definition \ref{Conjugate pairs} is equivalent to this. We use Definition \ref{Conjugate pairs} because it is more convenient in the sequel. In the following lemma we prove that Definition \ref{Conjugate pairs} is equivalent to the traditional definition of conjugate points.

\begin{lemma}
If $(v,\tilde{v})$ is a conjugate pair of order $k$ (as per Definition \ref{Conjugate pairs}), then there is a $k$ dimensional space of Jacobi fields along $\gamma_v$ vanishing at both $\pi(v)$ and $\pi(\tilde{v})$. Conversely, if $\gamma_v$ passses through $\pi(\tilde{v})$ and there is a $k$ dimensional space of Jacobi fields along $\gamma_v$ vanishing at $\pi(v)$ and $\pi(\tilde{v})$, then $(v,\tilde{v})$ is a conjugate pair of order $k$.
\end{lemma}
\begin{proof}
Suppose first that $(v,\tilde{v})$ is a conjugate pair of order $k$, and take
\[
X \in K_{(v,\tilde{v})} =\mathrm{ker} \Big ( D \pi|_{\tilde{v}} \circ D_v \Psi |_{(v,\tau_+(v) - \tau_+(\tilde{v})} \Big ) \bigcap \mathrm{ker} \Big ( D \pi |_{v} \Big ).
\]
Suppose that $\alpha$ is a curve in $S M$ such that
\[
X = \dot \alpha(0),
\]
and that we have some natural coordinates $(x^j,v^j)$ on $TM$ such that in these coordinates
\[
X = a^j \frac{\partial}{\partial v^j}.
\]
Using $X$ we define a vector field $J$ along $\gamma_v$ by
\be{JX}
J(\gamma_v(t)) = D \pi |_{\Psi(v,t)} \circ D_v \Psi |_{(v,t)}\ X.
\ee
Clearly $J$ vanishes at $\pi(v)$ and $\pi(\tilde{v})$, and $J$ is a Jacobi field since it is the variation field of the variation through geodesics defined by
\be{gamma}
\gamma(t,s) = \pi(\Psi(\alpha(s),t)).
\ee
Finally, to check that the mapping $X \mapsto J$ is injective (and so there is a $k$ dimensional space of vanishing Jacobi fields), a calculation in coordinates, using the fact that $J$ vanishes at $\pi(v)$, shows that
\be{JXcoord}
\nabla_{v} J(\pi(v)) = a^j \frac{\partial}{\partial x^j}.
\ee
Here $\nabla$ is the Levi-Civita connection given by the metric $g$. Thus the map from $X$ to $\nabla_v J(\pi(v))$ is injective, and so the map $X \mapsto J$ is also injective.

Conversely suppose that $\gamma_v$ passes through $\pi(\tilde{v})$, and there is a $k$ dimensional space of Jacobi fields along $\gamma_v$ vanishing at $\pi(v)$ and $\pi(\tilde{v})$. Let $J$ be any such vanishing Jacobi field and let $\gamma(t,s)$ be a variation through unit speed geodesics with variation field $J$. The geodesics in the variation can be chosen to be unit speed since $J$ must vanish at two different points. Then define
\[
X = \frac{\partial^2 \gamma}{\partial t \partial s}(0,0) \in T SM
\]
so that in particular $X = \dot \alpha(0)$ where $\alpha(s) = \frac{\partial \gamma}{\partial t}(0,s)$. Then \eqref{gamma} holds, and from this we can then see that $X$ and $J$ are related by \eqref{JX}. This completes the proof.
\end{proof}

We now study the structure of the sets $C_{R,k}$ by adapting the method of \cite{Warner65} to prove the following theorem.

\begin{theorem}\label{Jrktheorem}
For each $k$ the set $C_{R,k}$ is an embedded $2n-1$ dimensional submanifold of $S M^{int} \times S M^{int}$. Furthermore, the set
\[
\begin{split}
J_{R,k} & := \{ ((v,\tilde{v}), (X,\tilde{X})) \in T(SM^{int} \times SM^{int}) \ : \ (v,\tilde{v}) \in C_{R,k}, \ X \in K_{(v,\tilde{v})}, \\
& \tilde{X} = D_v \Psi|_{(v,\tau_+(v)-\tau_+(\tilde{v}))} X \}
\end{split}
\]
forms a smooth vector bundle of dimension $k$ over $C_{R,k}$.
\end{theorem}

\begin{proof}
The proof of the first point is an extension of methods found in \cite{Warner65}.  Indeed, let us begin by considering the exponential map $\mathrm{exp}(w) = \gamma_{w}(1): T M^{int} \supset \mathcal{D}_{\mathrm{exp}} \rightarrow M^{int}$. Let $D_F \mathrm{exp}$ be the differential of this map restricted to the tangent space of the fibres. Note carefully that we are considering here the exponential map acting on $T M^{int}$, but this differential is only in the fibre variables.

Now suppose that $(v,\tilde{v}) \in C_{R,k}$. By the results in \cite{Warner65} there exist coordinates on a neighborhood $W$ of $w = (\tau_+(v) - \tau_+(\tilde{v}))v \in T M^{int}$ and a neighborhood $V$ of $\pi(\tilde{v}) = \mathrm{exp}(w)$ such that the derivative in the radial direction of the $k-1$st elementary symmetric function applied to the eigenvalues of $D_F \mathrm{exp}$ expressed in these coordinates, which we will label $\sigma_{k-1}$, does not vanish. This implies that $\sigma_{k-1}^{-1}(0)$ corresponds, through the coordinate map, with an embedded $2n-1$ dimensional submanifold of $\mathcal{D}_{exp}\setminus\{0\}$ contained in $W$ with the relative topology. Note that $\sigma_{k-1}^{-1}(0)$ is precisely the set of vectors in $W$ with conjugate points of order $k$ or higher in $V$. Now let $U$ be a neighborhood of $(v,\tilde{v})$ in $S M^{int} \times S M^{int}$ such that all other conjugate pairs in $U$ also have order $k$ and shrink $U$ if necessary so that for all $(w,\tilde{w}) \in U$, $\lambda w \in W$ for some $\lambda \in \mathbb{R}^+$.

Next consider the smooth map
\[
p: \mathcal{D}_{exp} \setminus \{0\} \rightarrow S M^{int} \times S M^{int}
\]
defined by
\[
p(w) = \left ( \frac{w}{|w|_g}, \Psi\left ( \frac{w}{|w|_g}, |w|_g \right ) \right ).
\]
From our construction we have that 
\[
p(\sigma_{k-1}^{-1}(0) \cap p^{-1}(U)) = C_{R,k} \cap U.
\]
Also, $p|_{\sigma_{k-1}^{-1}(0) \cap p^{-1}(U)}$ is invertible with inverse defined on its image by
\[
G(v,\tilde{v}) = (\tau_+(\tilde{v}) - \tau_+(v)) v.
\]
This formula shows that the inverse may be extended to a continuous function on all of $S M^{int} \times S M^{int}$ and therefore that $p|_{\sigma_{k-1}^{-1}(0) \cap p^{-1}(U)}$ is a topological embedding. Since we have already shown that $\sigma_{k-1}^{-1}(0) \cap p^{-1}(U)$ is an embedded submanifold of $\mathcal{D}_{exp} \setminus \{0\}$ we also have that $p|_{\sigma_{k-1}^{-1}(0) \cap p^{-1}(U)}$ is a smooth mapping and thus combining all of this we conclude that $C_{R,k} \cap U$ is an embedded submanifold of $S M^{int} \times S M^{int}$. Since $(v,\tilde{v}) \in C_{R,k}$ was chosen arbitrarily this implies that in fact $C_{R,k}$ is itself an embedded submanifold of $S M^{int} \times S M^{int}$.

For the second point we simply comment that $J_{R,k}$ can be identified as the sub-bundle of another vector bundle $V_{R,k}$ over $C_{R,k}$ obtained by pulling back the vertical sub-bundle $V$ of $T(SM^{int} \times SM^{int})$ ($V = \mathrm{ker}(D\pi) \times \mathrm{ker}(D \pi)$). With this set-up, $J_{R,k}$ is the kernel of the bundle mapping
\[
V_{R,k}\ni (X,\tilde{X}) \mapsto \tilde{X} - D_v \Psi|_{(v,\tau_+(v) - \tau_+(\tilde{v}))} X \in T(SM^{int} \times SM^{int}).
\]
Since this mapping is constant rank $k$ the kernel is a sub-bundle of $V_{R,k}$ as claimed. This completes the proof.
\end{proof}

\noindent Now we introduce a bundle mapping defined on $J_{R,k}$ that will be used to describe the canonical relation of the normal operator.


\begin{lemma} \label{lem:C_v,v}
For each $k$ there is a bundle mapping
\[
\mathcal{C}_k: J_{R,k} \rightarrow T^*(M^{int} \times M^{int}) = T^* M^{int} \times T^* M^{int}
\]
defined by the requirement that for $((v,\tilde{v}), (X,\tilde{X})) \in J_{R,k}$
\[
(\mathfrak{i}_{D i_{SM}|_v X} \omega_g, \mathfrak{i}_{D i_{SM}|_{\tilde{v}} \tilde{X}} \omega_g) = D \pi_{T(M^{int} \times M^{int})}|_{(v,\tilde{v})}^t \mathcal{C}_k \big (((v,\tilde{v}), (X,\tilde{X})) \big ).
\]
\end{lemma}

\begin{proof}
First we note that for any $v \in T M^{int}$, the kernel of $D \pi_{TM} |_v$, which is the tangent space of the fibre, is a Lagrangian subspace of $T_v (T M^{int})$ with respect to the symplectic form $\omega_g$. This can be seen from the fact that $\flat_g$ is a bundle isomorphism. Indeed
\[
\pi_{TM} = \pi_{T^*M} \circ \flat_g \Rightarrow D\pi_{TM}|_v =  D \pi_{T^*M}|_{\flat_g(v)} \circ D\flat_g|_v
\]
implies that $D \flat_g|_v$ maps the kernel of $D\pi_{TM}|_v$ into the kernel of $D \pi_{T^*M}|_{\flat_g(v)}$. Since $D \flat_g|_v$ invertible, and these kernels have the same dimension, in fact $D \flat_g|_v$ restricted to the kernel of $D \pi_{TM} |_v$ is an isomorphism onto the kernel of $D \pi_{T^*M} |_{\flat_g(v)}$ which is a Lagrangian subspace with respect to the canonical symplectic form $\omega$. Since $\flat_g$ is a symplectomorphism by definition of $\omega_g$ it follows that $\mathrm{ker}(D \pi_{TM} |_v)$ is a Lagrangian subspace. This point can also be shown using coordinates and the formula \eqref{omegag} below.

Now, for $((v,\tilde{v}), (X,\tilde{X})) \in J_{R,k}$ we have that
\[
D \pi_{T M} |_{v} \circ D i_{SM}|_v \ X = 0\ \mbox{and}\ D \pi_{TM}|_{\tilde{v}} \circ D i_{SM}|_{\tilde{v}}  \ \tilde{X}=0.
\]
The remainder of the proof is the same for the two cases, and so we just consider the first. Indeed, let $(x^i, v^j)$ be natural local coordinates on $TM$ in a neighborhood of $v$. Calculation in these coordinates reveals
\[
\begin{split}
\left [\mathfrak{i}_{D i_{SM}|_v X} \omega_g\right ]\left ( a^i \frac{\partial}{\partial x^i} + b^j \frac{\partial}{\partial v^j} \right ) & = \omega_g \left (D i_{SM}|_v\ X,\ a^i \frac{\partial}{\partial x^i} + b^j \frac{\partial}{\partial v^j} \right )\\
& = \omega_g \left ( D i_{SM}|_v\ X,\ a^i \frac{\partial}{\partial x^i} \right ).
\end{split}
\]
In the second equality we have used the fact that $\mathrm{ker}(D \pi_{TM} |_v)$ is Lagrangian and $D i_{SM}|_v\ X \in \mathrm{ker}(D \pi_{TM} |_v)$. From this we have
\[
\mathfrak{i}_{D i_{SM}|_v X} \omega_g = \omega_g \left ( D i_{SM}|_v\ X, \frac{\partial}{\partial x^i} \right )\ \mathrm{d} x^i \in \mathrm{image}(D \pi_{TM} |_v^t).
\]
This shows that a map is defined by the requirement given in the lemma, and in fact from this last formula we can see that it depends linearly on the fibre variables in $J_{R,k}$ and smoothly on all the variables. Therefore the proof is complete.
\end{proof}

\noindent It is worth noting that the map $\mathcal{C}_k$ can also be expressed in at least two other ways, and in fact has a rather simple expression in coordinates. Indeed, if $(x^i,v^i)$ are natural coordinates on $TM$,
\[
D i_{SM}|_v \ X = a^i \frac{\partial}{\partial v^i},
\]
and
\[
v = \xi^j \frac{\partial}{\partial v^j},
\]
then in coordinates we can calculate
\be{omegag}
\omega_g = \xi^l \frac{\partial g_{il}}{\partial x^j}\ dx^j \wedge dx^i + g_{ij}\ dv^j \wedge dx^i.
\ee
From this we find that if $(\eta,\tilde{\eta}) = \mathcal{C}_k((v,\tilde{v}),(X,\tilde{X}))$, then in the corresponding coordinates $x^i$ on $M$
\[
\eta = a^i g_{ik} dx^k.
\]
With this in mind, we can see that $\eta$ is also obtained by applying the inverse of the vertical lift to $D i_{SM}|_v X$, followed by $\flat_g$. Another equivalent coordinate invariant definition is to take the Jacobi field $J$ defined in \eqref{JX}, and then $\eta=\flat_g \nabla_v J(\pi(v))$.

We will prove one more lemma regarding the structure of conjugate points, which is really the geometric heart of the proof of our main result.

\begin{lemma} \label{etalem}
Let $v$ and $\tilde{v} \in S M^{int}$, $\eta \in T^*_{\pi(v)} M$, $\tilde{\eta} \in T^*_{\pi(\tilde{v})} M$, $\xi \in T^*_{F(v)} \partial_- S M$ be such that $F(v) = F(\tilde{v})$,
\be{Ccond}
D \pi |_v^t \eta = D F|_v^t \xi,\ \mbox{and} \ D \pi |_{\tilde{v}}^t \eta = D F|_{\tilde{v}}^t \xi.
\ee
Then $\eta(v) = 0$, $\tilde{\eta}(\tilde{v}) = 0$. If $v \neq \tilde{v}$ then the pair $(v, \tilde{v})$ is conjugate, and if $(v,\tilde{v}) \in C_{R,k}$ then $(\eta,\tilde{\eta}) \in \mathcal{C}_k(J_{R,k})$. Conversely if $(\eta,\tilde{\eta}) \in \mathcal{C}_k(J_{R,k})$ then there is a $\xi \in T^*_{F(v)} \partial_- S M$ such that \eqref{Ccond} holds.
\end{lemma}

\begin{proof}
We first observe that since $F(\dot \gamma_v(s))$ is constant with respect to $s$
\[
DF|_v ( \ddot \gamma_v(0)) = 0
\]
and combining this with the hypotheses we see that
\[
\eta(v) = D \pi |_v^t \eta \big (\ddot \gamma_v(0) \big ) = D F|_v^t \xi\big (\ddot \gamma_v(0)\big ) = \xi\big (DF|_v ( \ddot \gamma_v(0)) \big ) = 0,
\]
and similarly $\tilde{\eta}(\tilde{v}) = 0$. This proves the first assertion.

Next assume that $v \neq \tilde{v}$. Since $F(v) = F(\tilde{v})$ if we set $s = \tau_+(v) - \tau_+(\tilde{v})$, then $\Psi(v,s) = \tilde{v}$. Therefore
\be{DFeq}
F(v) = F(\Psi(v,s)) \Rightarrow DF|_v = DF|_{\tilde{v}} \circ D_v \Psi |_{(v,s)},
\ee
and so for $\xi \in T^*_{F(v)} \partial_- S M$
\[
DF|_v^t \xi = D_v \Psi |_{(v,s)}^t \circ DF|_{\tilde{v}}^t \xi
\]
which implies
\be{etaeq}
D \pi |_v^t \eta = D_v \Psi |_{(v,s)}^t \circ D \pi |_{\tilde{v}}^t \tilde{\eta}.
\ee
Now recall that $i_{SM}: SM^{int} \rightarrow T M^{int}$ is the inclusion mapping. Then we have
\[
\pi = \pi_{TM} \circ i_{SM} \ \mbox{and} \ i_{SM}(\Psi(v,s)) = \tilde{\Psi}(i_{SM}(v),s)
\]
which implies
\be{eq}
D \pi |_v^t = D i_{SM}|_v^t \circ D \pi_{TM}|_{v}^t \ \mbox{and} \ D \Psi |_{(v,s)}^t \circ D i_{SM} |_{\tilde{v}}^t= D i_{SM}|_v^t \circ D_v \tilde{\Psi}|_{(v,s)}^t.
\ee
Therefore, using the fact that the kernel of $D i_{SM}|_v^t$ is the span of the differential of $w \mapsto | w |^2_g$, \eqref{etaeq} implies that
\[
D \pi_{TM} |_{v}^t \eta = D_v \tilde{\Psi} |_{(v,s)}^t \circ D \pi_{TM} |_{\tilde{v}}^t \tilde{\eta} + \alpha\ \mathrm{d} \left ( \frac{|w|^2_g}{2} \right )
\]
for some $\alpha$. Applying the covectors on each side of this equation to a radial vector $r$ based at $v$ we obtain
\[
0 = \tilde{\eta}\left (D \pi_{TM} |_{\tilde{v}} \circ D_v \tilde{\Psi} |_{(v,s)} r \right ) + \alpha.
\]
Since $D \pi_{TM} |_{\tilde{v}} \circ D_v \tilde{\Psi} |_{(v,s)} r$ is parallel to $\tilde{v}$ the first assertion in the lemma implies that $\alpha = 0$ and therefore we conclude that
\be{etaeq2}
D \pi_{TM} |_{v}^t \eta = D_v \tilde{\Psi} |_{(v,s)}^t \circ D \pi_{TM} |_{\tilde{v}}^t \tilde{\eta}.
\ee
We now claim that this implies that $(v,\tilde{v})$ is a conjugate pair.

To prove this claim we will set
\[
X = (D \pi_{TM} |_{v}^t \eta)^{\sharp_{\omega_g}}.
\]
We first show that $X$ is in the range of $D i_{SM}$. To do this, take any $Y \in \mathrm{ker}(D \pi_{TM} |_v)$. Then 
\[
0 = \eta(D \pi_{TM} |_v Y) = D \pi_{TM} |_v^t \eta(Y)  = \omega_g(X,Y).
\]
As part of the proof of Lemma~\ref{lem:C_v,v} we showed that $\mathrm{ker}(D \pi_{TM} |_v)$ is Lagrangian, and so this implies that $X$ is in $\mathrm{ker}(D \pi_{TM}|_v)$. Next take a natural coordinate system $(x^i,v^i)$ on $TM$ in a neighbourhood of $v$ such that $g_{ij} = \delta_{ij}$ at $\pi(v)$, and
\be{coordv}
v = D \pi_{TM} |_{v} \frac{\partial}{\partial x^1}.
\ee
Then since $X \in \mathrm{ker}(D\pi_{TM}|_v)$, in these coordinates
\[
X = a^j \frac{\partial}{\partial v^j}
\]
for some coefficients $a^j$. We thus have
\[
0 = \eta(v) = D \pi_{TM}|_v^t \eta \left (  \frac{\partial}{\partial x^1} \right ) = \omega_g\left (X, \frac{\partial}{\partial x^1}\right ).
\]
Using \eqref{omegag} we finally see that
\[
0 = a^1,
\]
and so $X$ is in the range of $D i_{SM}$. Thus we can define $D i_{SM}|_v^{-1} X \in T SM^{int}$. Now we will show below that
\be{XK}
X \in \mathrm{ker} \Big ( D \pi_{TM} |_{\tilde{v}} \circ D_v \tilde{\Psi} |_{(v,s)} \Big ) \bigcap \mathrm{ker} \Big ( D \pi_{TM} |_v \Big ).
\ee
By \eqref{eq}, after taking transposes, \eqref{XK} implies that $D i_{SM}|_v^{-1} X \in K_{v,\tilde{v}}$ thus proving that $(v, \tilde{v})$ is a conjugate pair.

Let us begin proving \eqref{XK}. We have already shown above that $X \in \mathrm{ker}(D \pi_{TM}|_v)$. To finish, take any $Y \in \mathrm{ker}(D \pi_{TM} |_{\tilde{v}})$. Then similarly to before using also \eqref{etaeq2} in the second equality we have
\be{eq2}
0 = D \pi_{TM} |_{\tilde{v}}^t \tilde{\eta}\ (Y) = D \pi_{TM} |_v^t \eta\ (D_v \tilde{\Psi} |_{(v,s)}^{-1} Y) = \omega_g \left ( X, D_v \tilde{\Psi} |_{(v,s)}^{-1} Y \right ) .
\ee
Finally, since $v \mapsto \tilde{\Psi}(v,s)$ is a symplectomorphism we have
\be{eq3}
0 = \omega_g\left ( D_v \tilde{\Psi} |_{(v,s)}\ X, Y \right ).
\ee
Therefore $D_v \tilde{\Psi} |_{(v,s)} X$ is in the kernel of $D \pi_{TM} |_{\tilde{v}}$ which completes the proof of \eqref{XK}. Note that we have actually shown that $(D i_{SM}|_v^{-1}\ X,D i_{SM}|_v^{-1} \circ D\tilde{\Psi}|_{(v,t)} \ X) \in K_{(v,\tilde{v})}$, and so noticing also that
\[
D \pi_{TM}|_{v}^t \eta = \mathfrak{i}_{X} \omega_g, \quad \mbox{and} \quad D \pi_{TM}|_{\tilde{v}}^t \tilde{\eta} = \mathfrak{i}_{D\tilde{\Psi}|_{(v,t)}  X} \omega_g,
\] 
(the second equality follows from a calculation quite similar to \eqref{eq2} and \eqref{eq3}) we see that if $(v,\tilde{v}) \in C_{R,k}$, then $(\eta,\tilde{\eta}) \in \mathcal{C}_k(J_{R,k})$.

Now we assume that $(\eta,\tilde{\eta}) = \mathcal{C}_k((v,\tilde{v}),(X,\tilde{X}))$ for some $((v,\tilde{v}),(X,\tilde{X})) \in J_{R,k}$. To complete the proof we must show that there exists $\xi \in T^*_{F(v)} \partial_- SM$ such that \eqref{Ccond} is satisfied. We split the proof into several steps starting by showing that $\eta(v) = \tilde{\eta}(\tilde{v}) = 0$.

For this first step we take the same natural coordinate system $(x^i,v^i)$ on $TM$ as described above \eqref{coordv}. Then since $X \in \mathrm{ker}(D\pi|_v)$, in these coordinates
\[
Di_{SM}|_v\ X = \sum_{j=2}^n a^j \frac{\partial}{\partial v^j}
\]
for some coefficients $a^j$. We have
\[
\eta(v) = D\pi|_v^t \eta \left ( \frac{\partial}{\partial x^1} \right ) = \mathfrak{i}_{Di_{SM}|_v X}\omega_g \left ( \frac{\partial}{\partial x^1} \right ) = \omega_g \left (Di_{SM}|_v \ X, \  \frac{\partial}{\partial x^1} \right ).
\]
Using \eqref{omegag} this gives proves that $\eta(v) = 0$. The same argument shows that $\tilde{\eta}(\tilde{v}) = 0$.

Now, from the first paragraph of this proof we see that $\beta \in \mathrm{range}(D F|_v^t)$ if and only if $\beta (\ddot \gamma_v(0)) = 0$. Since $D\pi|_v^t \eta (\ddot \gamma_v(0)) = \eta (v) = 0$, we thus conclude that there exists $\xi \in T^*_{F(v)} \partial_- SM$ such that $D\pi|_v^t \eta =  D F|_v^t \xi$. The same reasoning shows that there exists $\tilde{\xi} \in T^*_{F(v)} \partial_- SM$ such that $D\pi|_{\tilde{v}}^t \tilde{\eta} =  D F|_{\tilde{v}}^t \tilde{\xi}$. It remains to show that $\xi = \tilde{\xi}$.

For this last step, we begin by noting that a lengthy but routine calculation similar in flavor to \eqref{eq2}, and making use of the fact that $\tilde{X} = D_v \Psi|_{(v,s)}\ X$, shows that
\[
D \pi|_v^t \eta =  D_v \Psi|_{(v,s)}^t \circ D \pi|_{\tilde{v}}^t \tilde{\eta} \Rightarrow DF|_v^t\ \xi = D_v \Psi|_{(v,s)}^t \circ D F|_{\tilde{v}}^t \ \tilde{\xi}.
\]
Finally using \eqref{DFeq} (tranposed) this implies
\[
DF|_v^t\ \xi = DF|_v^t\ \tilde{\xi},
\]
and so the injectivity of $DF|_v^t$ completes the proof.
\end{proof}

\noindent We have now reached our main theorem characterizing the structure of the normal operator $\mathcal{N}$. In the statement we use the term {\it local canonical relation} to describe the canonical relations $C_{A_k}$. This means that they are each the image of a constant rank map $\pi_k: \hat{C}_k \rightarrow C_{A_k}$ such that for any $c \in \hat{C}_k$ there is a neighborhood $U$ of $c$ such that $\pi_k(U)$ is a canonical relation. 

\begin{theorem} \label{mainthm}
Suppose that $C_S = \emptyset$. Then the sets
\[
C_{A_k} = \mathcal{C}_k(J_{R,k}) \subset T^*(M^{int} \times M^{int})
\]
are either empty or are local canonical relations. On the level of operators, if $\phi \in C^\infty(SM)$ is greater than or equal to zero everywhere and $\mathcal{N}_\phi$ is defined by \eqref{Ndef}, then we have a decomposition
\[
\mathcal{N}_\phi = \Upsilon + \sum_{k=1}^{n-1} \Bigg ( \sum_{m=1}^{M_k} A_{k,m} \Bigg )
\]
where $\Upsilon$ is a pseudodifferential operator of order $-1$, and for each $k$ either
\[
A_{k,m} \in \mathcal{I}^{-(n-k+1)/2} \ (M^{int} \times M^{int}, C'_{A_{k,m}}; \Omega_{M^{int}\times M^{int}}^{1/2})
\]
where $C_{A_{k,m}} \subset C_{A_k}$ for each $m$, or $M_k = 1$ and $A_{k,1} = 0$ if $C_{A_k} = \emptyset$. Furthermore, $\Upsilon$ is elliptic at every point $\eta \in T^* M^{int}$ such that there exists a $v \in SM^{int}$ with $\eta(v) = 0$ and $\phi(v) \neq 0$. 
\end{theorem}
\begin{remark} \label{singremark}
In dimension two, there can only be conjugate points of order one and so this theorem covers all possibilities in that case. However in dimension three or higher the generic case includes singular conjugate pairs. Indeed, according to \cite{Arnold73} and \cite{Klok83}, using the notation of \cite{Arnold73}, singularities of type $D_4$ occur generically in the exponential map in three dimensions or higher, and correspond to singular conjugate pairs of order two.
\end{remark}

\begin{proof}
The clean composition calculus does not apply directly to \eqref{N2comp}, but as we will see we can apply a partition of unity so that the calculus applies to each of the separated pieces. Let us begin by looking at the composition of the canonical relations for the operators in \eqref{N2comp} using Lemma \ref{pushpull_lemm} and \eqref{LambdaL}
\be{Ncanon}
\begin{split}
C_{\pi_*} \circ C_{L} = \Big \{ (\eta, \tilde{\eta}) \in T^* M^{int} \times T^* M^{int} \ & : \ \exists\ (v,\tilde{v}) \in SM^{int} \times SM^{int} \\
& \mbox{such that} \ F(v) = F(\tilde{v}), \ \exists \ \xi \in T_{F(v)}^* \partial_- S M^{int} \\
& \mbox{such that} \ D \pi |_v^t \eta = D F|_v^t \xi, \ D \pi |_{\tilde{v}}^t \eta = D F|_{\tilde{v}}^t \xi  \Big \}.
\end{split}
\ee
We will split this composition into different pieces corresponding to different orders of conjugate points. First we note that one piece is the diagonal
\[
\Delta = \{ (\eta, \eta) \in T^* M^{int} \times T^* M^{int} \}
\]
which can be seen to be contained in $C_{\pi_*} \circ C_L$ by taking $\tilde{v} = v$ such that $\eta(v) = 0$ in \eqref{Ncanon}.

Now suppose we have $(\eta, \tilde{\eta}) \in C_{\pi_*} \circ C_L$ with $\eta \neq \tilde{\eta}$. Suppose $v$, $\tilde{v}$ and $\xi$ satisfy the requirements in \eqref{Ncanon}. Then since $D F|_v^t$ and $D \pi|^t_v$ are injective we must have $v \neq \tilde{v}$ and we may use the result of Lemma~\ref{etalem} and the hypothesis $C_S = \emptyset$ to conclude that
\[
C_{\pi_*} \circ C_L = \Delta \cup \left ( \bigcup_{k=1}^{n-1} C_{A_k} \right ).
\]
This union may not be disjoint in general, but we will show that we can still decompose the operator $\mathcal{N}_\phi$ into a sum of FIOs each having a canonical relation contained in one of the sets in the union. 

We will take the convention that $C_{R,n} = \{(v,v) \in S M^{int} \times S M^{int}\}$. By the hypothesis that $C_S = \emptyset$ and the definition of $C_{R,k}$ there exists a collection of open subsets $U_k$ of $S M \times S M$ with disjoint closures such that for each $k =1$ to $n$
\[
C_{R,k} \subset U_k.
\]
Thus it is possible to construct a partition of unity $\{\phi_k\}_{k=0}^n$ on $S M \times S M$ such that $\mathrm{supp}(\phi_k) \subset U_k$ for $k =1$ to $n$ and
\[
\mathrm{supp}(\phi_0) \subset (S M \times S M) \setminus \left ( \bigcup_{k=1}^n C_{R,k} \right ).
\]
Since this is a partition of unity
\[
\mathcal{N}_\phi^\Phi = \sum_{k=0}^n\pi_* \circ L_{(\phi\otimes\phi)\phi_k}.
\]
By our construction, for each $k$ the wavefront set of the operator $L^\Phi_{(\phi\otimes\phi)\phi_k}$ is contained in the set
\[
\begin{split}
C_{L,k}' = \{ (DF|_{v}^t \xi, -\tilde{\eta}) \in T^* SM^{int} \times T^* M^{int} \ & : \ v 
\in S M^{int}, \ \exists\ \tilde{v} \in S M^{int} \\
&\hskip-2.25in \mbox{such that} \ (v,\tilde{v}) \in \mathrm{supp}(\phi_k),\ \pi(\tilde{v}) = \pi_{T^* M}(\tilde{\eta}), \ F(v) = F(\tilde{v}), \ DF|_{\tilde{v}}^t \xi = D\pi |_{\tilde{v}}^t \tilde{\eta} \}.
\end{split}
\]
Therefore, when we microlocalise in the proof of the clean composition calculus (see \cite{HormanderIV}) we obtain smoothing operators except near points in $C_{\pi_*} \circ C_{L,k}$. For $k = 0$ this composition is empty, and so we obtain a smoothing operator in that case and this is included in $\Upsilon$. For $k = n$ the composition is the diagonal $\Delta$, and finally for $k = 1$ to $n-1$ we have  $C_{\pi_*} \circ C_{L,k} = C_{A_k}$. We will now show that each of these compositions is clean.

We'll first consider the composition $C_{\pi_*} \circ C_{L,n} = \Delta$. Based on our constructions so far using the results of Lemma~\ref{etalem} we have the following
\[
\begin{split}
\hat{C}_n  = \left ( C_{\pi_*} \times C_{L,n} \right ) & \bigcap \left ( T^* M^{int} \times \Delta( T^* S M^{int} ) \times T^* M^{int} \right ) \\
&=  \Big \{ (\eta, D \pi |_v^t \eta , D \pi |_v^t \eta, \eta)  \ : \ v \in S M^{int}, \ \eta \in T_{\pi(v)}^*M^{int}, \ \eta(v) = 0 \Big \}.
\end{split}
\]
Using the fact that $\mathcal{O}$ introduced in \eqref{mcO} is an embedded submanifold, we easily see that $\hat{C}_n$ is also embedded submanifold of dimension $3n-2$. Furthermore we can see that the intersection is clean in the following way. Suppose that $\delta = (\delta_1,\delta_2,\delta_3,\delta_4) \in \hat{C}_n$ and
\[
(Y_1,Y_2,Y_3,Y_4) \in T_{\delta} \left (C_{\pi_*} \times C_{L,n} \right ) \bigcap T_{\delta} \left ( T^* M^{int} \times \Delta( T^* S M^{int} ) \times T^* M^{int} \right ) =: D_n
\]
where we are considering this set $D_n$ as a subset of
\[
T_{\delta_1}(T^*M^{int}) \times T_{\delta_2}(T^* S M^{int}) \times T_{\delta_3}(T^* S M^{int}) \times  T_{\delta_4}(T^*M^{int}).
\]
Then we must have $Y_2 = Y_3$, and based on this and an examination of $C_{L,n}$ we find that in fact $Y_3$ and $Y_4$ are determined by $Y_2$. Therefore the dimension of $D_n$ is at most the dimension of $C_{\pi_*}$ which is $3n-1$. However, we can observe that the requirement $\eta(v) = 0$, which holds because of the requirement $DF|_{\tilde{v}}^t \xi = D\pi |_{\tilde{v}}^t \tilde{\eta}$ in the definition of $C_{L,n}$ and the first part of the proof of Lemma \ref{etalem}, eliminates one more dimension and so in fact the dimension of $D_n$ is at most $3n-2$. Since necessarily $T_{\delta} \hat{C}_n \subset D_n$ and the dimension of $T_\delta \hat{C}_n$ is $3n-2$ we conclude that the intersection must be clean with excess $e = n-2$. It is also easy to see that the projection map $\pi_n: \hat{C}_n \rightarrow T^* M^{int} \times T^* M^{int}$ is proper and the fibres of this map are exactly
\[
\pi_n^{-1}(\eta,\eta) = \Big \{ (\eta, D \pi |_v^t \eta ; D \pi |_v^t \eta, \eta)  \ :\ \eta(v) = 0 \Big \}
\]
which are certainly connected. The clean composition calculus thus implies that $(\pi_* \circ L_{(\phi\otimes\phi)\phi_n}$ is a pseudodifferential operator of order $(1-n)/4 - (n+1)/4 + (n-2)/2  = -1$ which is $\Upsilon$, modulo a smoothing operator, in the decomposition given in our theorem. We comment that the excess of the clean intersection corresponds with those $v \in S M^{int}$ such that $\eta(v) = 0$, and so a principal symbol for $\Upsilon$ can be found by integrating the product of symbols for $\pi_*$ and $L_{(\phi \otimes \phi)\phi_n}$ over this set. This agrees with formulae which have been found for a principal symbol of the normal operator in the past for simpler cases, and proves the statement in the theorem on the ellipticity of $\Upsilon$.

Now we turn to the compositions corresponding to conjugate points. If any of the compositions $C_{\pi_*} \circ C_{L,k}$ are empty, then the corresponding operator is smoothing, and so can be absorbed into $\Upsilon$. If the composition is not empty, then as in the last case we have based on Lemma~\ref{etalem} that
\[
\begin{split}
\hat{C}_k :&= \left ( C_{\pi_*} \times C_{L,k}\right ) \bigcap \left ( T^* M^{int} \times \Delta( T^* S M^{int} ) \times T^* M^{int} \right )\\
 & = \Bigg \{ \Big ( \mathcal{C}_k^1((v,\tilde{v}),(X,\tilde{X})), D i_{SM}|_v^t\ \mathfrak{i}_{Di_{SM}|_v X} \omega_g, D i_{SM}|_{\tilde{v}}^t\ \mathfrak{i}_{Di_{SM}|_{\tilde{v}}\tilde{X}} \omega_g, \mathcal{C}_k^2((v,\tilde{v}),(X,\tilde{X})) \Big ) \\ 
 &\hskip1in : \ ((v,\tilde{v}),(X,\tilde{X})) \in J_{R,k} \Bigg \}.
\end{split}
\]
Here $\mathcal{C}_k^1$ and $\mathcal{C}_k^2$ are the first and second components of $\mathcal{C}_k$ and $i_{SM}: SM^{int} \rightarrow T M^{int}$ is the inclusion mapping. This is an embedded $2n+k-1$ dimensional submanifold which is parametrized as shown in the previous formula by $J_{R,k}$ since the mapping
\[
J_{R,k} \ni ((v,\tilde{v}),(X,\tilde{X})) \mapsto \left ( D i_{SM}|_v^t\ \mathfrak{i}_{Di_{SM}|_v X} \omega_g ,\ D i_{SM}|_{\tilde{v}}^t\ \mathfrak{i}_{Di_{SM}|_{\tilde{v}}\tilde{X}} \omega_g \right ) \in T^*(S M^{int} \times SM^{int})
\]
is an embedding. To see that this is an embedding, note that $J_{R,k}$ is a sub-bundle of $V_{r,k}$ defined at the end of the proof of Theorem \ref{Jrktheorem}, and that the mapping extended by the same formula to $V_{R,k}$ is a diffeomorphism onto the bundle $\mathcal{O}$, defined by \eqref{mcO}, pulled back to $C_{R,k}$. To see that it is a diffeomorphism, note that it is a bundle map over the identity, and then check in coordinates that it is invertible in each fibre. 

We can see that the intersection is clean by a dimension counting argument similar to the previous case. As before suppose that $\delta = (\delta_1,\delta_2,\delta_3,\delta_4) \in \hat{C}_k$ and
\[
(Y_1,Y_2,Y_3,Y_4) \in T_{\delta} \left ( \Lambda_{(\pi^*)^t} \times \Lambda_{L_{\phi_k}}\right ) \bigcap T_{\delta} \left ( T^* M^{int} \times \Delta( T^* S M^{int} ) \times T^* M^{int} \right ) =: D_k.
\]
Just as before we have $Y_2$ determines $Y_3$ and $Y_4$. In this case we lose $n-k$ dimensions from $\mathrm{dim}(C_{\pi_*})$ since for fixed $(v, \tilde{v})$ the set
\[
\{ D i_{SM}|_v^t\ \mathfrak{i}_{D_{i_SM}|_v X} \omega_g \ : \ ((v,\tilde{v}),(X,\tilde{X})) \in J_{R,k} \}
\]
is a $k$ dimensional vector space contained in
\[
\{ D \pi |_v^t \eta \ : \ \eta \in T^*_{\pi(v)} M^{int} \}.
\]
Thus, again keeping $v$ fixed, $\eta$ is restricted to be in only a $k$ dimensional subspace of the $n$ dimensional space $T_v M^{int}$. Therefore the dimension of $D_k$ is at most $3n-1 - (n-k) = 2n + k -1$ which is the dimension of $T_{\delta} \hat{C}_k$ and so as before we conclude that the intersection is clean with excess $k-1$. The projections
\[
\pi_k:\hat{C}_k \rightarrow T^* M^{int} \times T^*M^{int}
\]
are proper because the inverse images of sets bounded away from the edge of $T^* M^{int} \times T^*M^{int}$ are bounded away from the edge in the larger space. This implies that $C_{A_k}$ is a local canonical relation for each $k$ (provided $C_{{A}_k} \neq \emptyset$).

To apply the clean composition calculus as given in \cite{HormanderIV} it is necessary that the maps $\pi_k$ have connected fibres. However this is not true in general. Nonetheless by what we have already done the sets $C_{A_k}$, are local canonical relations. Thus the clean composition calculus still may be applied if local representations of the relevant operators are used and the decomposition in the statement can be achieved.
 \end{proof}

\noindent We have attempted to make Theorem~\ref{mainthm} as general as possible and in so doing sacrificed some clarity in the statement. By a slight modification of the proof we could have the following simpler corollary which still covers many cases of interest.

\begin{corollary}\label{order1cor}
Suppose in addition to the hypotheses of Theorem \ref{mainthm} that there are only conjugate pairs of order $1$ in $S M^{int} \times S M^{int}$, and no two points are conjugate along more than one geodesic. Then we have a decomposition
\[
\mathcal{N}_\phi = \Upsilon + A_1
\]
where $\Upsilon$ is a pseudodifferential operator of order $-1$ and 
\[
A_1 \in \mathcal{I}^{-n/2} \left (M^{int} \times M^{int}, C_{A_1}; \Omega^{1/2}_{M^{int} \times M^{int}} \right ).
\]
\end{corollary}
\begin{proof}
The proof is the same as the proof of Theorem~\ref{mainthm}, except we note that the additional hypothesis that no two points are conjugate along more than one geodesic implies that the map $\pi_1:\hat{C}_1 \rightarrow C_{A_1}$ is injective, and so $C_{A_1}$ is a canonical relation. Thus the clean composition calculus applies without further decomposition as was necessary in Theorem~\ref{mainthm}.
\end{proof}

\section{Application to inversion of the normal equation} \label{stability:sec}

We now turn to the problem of inverting $\mathcal{N}_\phi$, or given $g$ solving the normal equation
\be{normal}
\mathcal{N}_\phi[f] = g
\ee
for $f$. To improve the notation going forward let us write
\[
A = \sum_{k=1}^{n-1} \Bigg ( \sum_{m=1}^{M_k} A_{k,l} \Bigg ) 
\]
so that with the decomposition from Theorem~\ref{mainthm} equation \eqref{normal} becomes
\[
\Upsilon[f] + A[f] = g.
\]
One standard approach is to find an appropriate function space on which this equation is of Fredholm type. Indeed, if $A$ is lower order than $\Upsilon$ in some sense, we may expect this is possible. In the setting of Corollary~\ref{order1cor} when the dimension $n$ is at least $3$, $A = A_1$ will be lower order as an FIO, and if $C_{A_1}$ is a local canonical graph then in fact $A$ has appropriate mapping properties. It is not always true that $C_{A_1}$ is a canonical graph, see examples in \cite{StefanovUhlmann12}, and a more detailed study of cases in which $C_{A_1}$ is not a canonical graph may be an interesting direction for future research, but we will say no more about it here.

To move forward we will need to take $(\widetilde{M},\widetilde{g})$ to be a smooth extension with convex boundary of $(M,g)$ also satisfying Assumption \ref{basicassume}. Precisely this means $M \Subset \widetilde{M}^{int}$ and $\widetilde{g}|_{M} = g$. Such an extension can always be found, and related to the extension we have the restriction maps $R: H^s(\Omega_{\widetilde{M}^{int}}^{1/2}) \rightarrow H^s(\Omega_{M^{int}}^{1/2})$ which are continuous for $s \geq 0$. The adjoints of the restriction maps are the extension-by-zero maps $\mathfrak{i}:H^s(\Omega^{1/2}_{M^{int}}) \rightarrow H^s(\Omega^{1/2}_{\widetilde{M}^{int}})$, which are isometric embeddings for $s \leq 0$ by duality. In view of the comments in the previous paragraph we make the following additional assumption about the extension $(\widetilde{M},\widetilde{g})$.
\begin{assumption}\label{Extassume}
Assume the dimension $n$ is at least three, that all conjugate pairs in $S \widetilde{M}^{int} \times S \widetilde{M}^{int}$ are of order $1$, and that $C_{A_1}$ (see Corollary~\ref{order1cor}) is a local canonical graph.
\end{assumption}
\noindent It should be possible to construct an extension satisfying this assumption when the same is true up to the boundary of $S M \times SM$. Thus the assumption is eliminating the possibility that there are singular conjugate pairs, or places where $C_{A_1}$ is not a graph over $\partial M \times \partial M$. 

We will also need to take an intermediate extension $\widetilde{M}_1$ satisfying the same requirements as $\widetilde{M}$ and such that $M \Subset \widetilde{M}_1 \Subset \widetilde{M}$. We will make use of all of the same objects defined on $M$, also defined in the analogous manner on $\widetilde{M}$ and $\widetilde{M}_1$, although we will add a tilde for objects on $\widetilde{M}$, and also a subscript $1$ for objects defined on $\widetilde{M}_1$. Thus for example $\widetilde{F}: S \widetilde{M}^{int} \rightarrow \partial_- S \widetilde{M}$ will be the defined in the same way that $F$ was defined just above \eqref{Fdef}, but with $M$ replaced by $\widetilde{M}$. In the same way we have the mapping $\widetilde{F}_1: S\widetilde{M}_1^{int} \rightarrow \partial_- S \widetilde{M}_1$. 

Next we state the main theorem of this section.

\begin{theorem}\label{contthm}
If Assumption \ref{Extassume} is satisfied, $\phi \in C^\infty(SM)$ is greater than or equal to zero everywhere and for every $\eta \in T^* M^{int}$ there exists a $v \in SM^{int}$ with $\eta(v) = 0$ and $\phi(v) \neq 0$, then the kernel of $\mathcal{X}_\phi$ acting on $L^2(\Omega_{M}^{1/2})$ is at most finite dimensional and is contained in $C^\infty_c(\Omega_{M^{int}}^{1/2})$. Furthermore, if $\mathcal{F} \subset L^2(\Omega_{M}^{1/2})$ is a closed subspace complementary to the kernel of $\mathcal{X}_\phi$ then
\be{stab}
\| \mathcal{X}_\phi[f] \|_{L^2(\Omega_{\partial_- S M}^{1/2})} \sim \| f \|_{H^{-1/2}(\Omega_{M}^{1/2})} 
\ee
for all $f \in \mathcal{F}$.
\end{theorem}

We will provide the proof of this theorem, but first we make a remark concerning its significance. The equation \eqref{stab}, which shows stability and continuity of $\mathcal{X}_\phi$ from $\mathcal{F}$ with the $H^{-1/2}$ norm to $L^2$, is intended to match with the hypotheses required for convergence of regularizations in Hilbert scales as originally established in \cite{Natterer}. Indeed, using the result of \cite{Natterer} together with Theorem \ref{contthm} we have the following corollary which shows the convergence rate of Tikhonov regularised solutions to the true solution of the problem assuming that $\mathcal{X}_\phi$ is injective.

\begin{corollary} \label{Tikhonovcorollary}
Suppose that assumption \ref{Extassume} is satisfied and $\mathcal{X}_\phi$ is injective on $L^2(\Omega_M^{1/2})$. Let $q > 0$ and $p \geq (q-1/2)/2$ be given, and suppose $f_0 \in H^{q}(\Omega_M^{1/2})$, and $g \in L^2(\Omega_{\partial_- SM}^{1/2})$ satisfies
\[
\|g - \mathcal{X}_\phi[f_0] \|_{L^2(\Omega_{\partial_- SM}^{1/2})} \leq \epsilon.
\]
Then for $\omega$ appropriately chosen the unique solution of the Tikhonov regularised problem
\[
\mathrm{arg \ min}_{f \in H^s(\Omega_{\partial_- S M}^{1/2})} \Big \{ \| \mathcal{X}_\phi[f] - g_0 \|_{L^2(\Omega_{\partial_- SM}^{1/2})} + \omega \| f \|_{H^{p}(\Omega_M^{1/2})} \Big \}
\]
satisfies
\[
\|f - f_0 \|_{L^2(\Omega_M^{1/2})} \lesssim \epsilon^{\frac{2q}{1 + 2q}} \|f_0\|^{\frac{1}{1+2q}}_{H^{q}(\Omega_M^{1/2})}
\]
\end{corollary}

\noindent If the geodesic X-ray transform arises as the linearization of a nonlinear problem, as it might in travel time tomography, we also comment that the condition \eqref{stab} plays an important role in the analysis of convergence of regularisation methods for the nonlinear problem (see for example \cite{Neubauer,Qi}).

We also comment that the use of $\mathcal{X}_\phi$ rather than $\mathcal{X}$ allows the result to cover also the weighted geodesic ray transform including the limited data case incorporated by setting $\phi$ to be a cut-off function.

We now turn to the proof of Theorem \ref{contthm} which will be broken into a series of lemmas. The first lemma concerns the continuity of $\mathcal{N}_\phi$.

\begin{lemma}\label{Ncont}
Let Assumption~\ref{Extassume} be satisfied and suppose that $\phi \in C^\infty(S M)$. Then for $-1\leq s \leq 0$
\[
\mathcal{N}_\phi: H^s(\Omega^{1/2}_M) \rightarrow H^{s+1} (\Omega^{1/2}_M)
\]
continuously.
\end{lemma}
\begin{proof}
Suppose that $M_1$ is an open set such that $M \Subset M_1 \Subset \widetilde{M}^{int}$, and take $\chi \in C_c^\infty(M_1)$ such that $\chi(x) = 1$ for $x \in M$. Also, let $\widetilde{\phi}$ be a smooth extension of $\phi$ to $\widetilde{M}$. By Corollary \ref{order1cor} using assumption~\ref{Extassume} and the continuity properties of FIOs we have
\be{N1cont}
\chi^m \circ \widetilde{\mathcal{N}}_{\widetilde{\phi}} \circ \chi^m: H^s(\Omega^{1/2}_{\widetilde{M}}) \rightarrow H^{s+1}(\Omega^{1/2}_{\widetilde{M}}).
\ee
We also note that from the equation \eqref{explicitN} for $\mathcal{N}_\phi$, we can see that for $f$ supported in $M$
\[
R \circ \chi^m \circ \widetilde{\mathcal{N}}_{\widetilde{\phi}} \circ \chi^m \circ \mathfrak{i}[f] = \mathcal{N}_\phi[f].
\]
Thus we find that for $f \in C^\infty(M)$
\[
\begin{split}
\|\mathcal{N}_\phi[f] \|_{H^{s+1}(\Omega^{1/2}_M)} & = \| R \circ \chi^m \circ \widetilde{N}_{\widetilde{\phi}} \circ \chi^m \circ \mathfrak{i}[f] \|_{H^{s+1}(\Omega^{1/2}_M)}\\
& \leq C \| \chi^m \circ \widetilde{N}_{\widetilde{\phi}} \circ \chi^m \circ \mathfrak{i}[f] \|_{H^{s+1}(\Omega^{1/2}_{\widetilde{M}})}\\
& \leq C \| \mathfrak{i}[f] \|_{H^{s}(\Omega^{1/2}_{\widetilde{M}})}\\
& \leq C \| f \|_{H^{s}(\Omega^{1/2}_M)}
\end{split}
\]
where the constant $C>0$ may change at each step. The first inequality follows from the continuity of $R$ on $H^{s+1}(\Omega^{1/2}_{\widetilde{M}})$ when $s \geq -1$, the second from \eqref{N1cont}, and the third from the continuity of $\mathfrak{i}$ on $H^{s}(\Omega^{1/2}_M)$ when $s \leq 0$. This completes the proof.
\end{proof}

\noindent Next we study the continuity of $\mathcal{X}_\phi$ and $\mathcal{X}_\phi^t$. 

\begin{lemma} \label{Xcont}
Make the same assumptions as in Lemma~\ref{Ncont}. Then
\[
\mathcal{X}_\phi: H^{-1/2}(\Omega^{1/2}_M) \rightarrow L^2(\Omega_{\partial_- SM}^{1/2}),
\]
and
\[
\mathcal{X}^t_\phi: L^2(\Omega_{\partial_- SM}^{1/2}) \rightarrow H^{1/2}(\Omega^{1/2}_M)
\]
continuously
\end{lemma}

\begin{proof}
Let $f \in C^\infty \left (\Omega^{1/2}_{M} \right )$. Then using Lemma \ref{Ncont}
\[
\begin{split}
\| \mathcal{X}_\phi[f] \|_{L^2(\Omega_{\partial_- SM}^{1/2})}^2 & = \langle \mathcal{N}_\phi[f], f \rangle_{L^2(\Omega^{1/2}_M)}\\
& \leq C \| \mathcal{N}_\phi[f] \|_{H^{1/2}(\Omega^{1/2}_M)} \| f \|_{H^{-1/2}(\Omega^{1/2}_M)}\\
& \leq C \| f \|_{H^{-1/2}(\Omega^{1/2}_M)}^2.
\end{split} 
\]
As usual, the constant $C$ may change between steps. This proves the first portion of the result. The second part follows by duality. Indeed, for $f \in C^\infty(\Omega^{1/2}_{M})$ and $h \in C^\infty(\Omega^{1/2}_{\overline{\partial_- SM}})$
\[
\begin{split}
\langle \mathcal{X}_\phi^t[h], f \rangle_{L^2(\Omega^{1/2}_M)} & = \langle h, \mathcal{X}_\phi[f] \rangle_{L^2(\Omega^{1/2}_{\partial_- SM})}\\
&  \leq \|h\|_{L^2(\Omega_{\partial_- SM}^{1/2})} \| \mathcal{X}_\phi[f] \|_{L^2(\Omega_{\partial_- SM}^{1/2})} \\
& \leq C \|h\|_{L^2(\Omega_{\partial_- SM}^{1/2})} \| f \|_{H^{-1/2}(M)}.
\end{split}
\]
Taking the supremum over $f$ such that $\| f \|_{H^{-1/2}(M)} = 1$ we obtain the result.
\end{proof}

\noindent We now begin the proof of Theorem \ref{contthm}.

\begin{proof}[Proof of Theorem \ref{contthm}]
Let $\partial_\delta M = \{ x \in M \ : \ \mathrm{dist}(x,\partial M) < \delta \}$. Then by the result of \cite{UhlmannVasy12}, since $\partial M$ is strictly convex and compact, there exists a $\delta >0$ such that if $f \in \mathrm{Ker}(\mathcal{X}_\phi)$ acting on $L^2(\Omega^{1/2}_{M})$, then $f |_{\partial_\delta M} = 0$. Now introduce open sets $M_1$, $M_2$, and $M_3$ such that
\[
\partial_\delta M^c \subset M_1 \Subset M_2 \Subset M_3 \Subset M^{int}.
\]
Also, let $\chi_1 \in C_c^\infty(M_2)$ and $\chi_2 \in C_c^\infty(M_3)$ be such that $\chi_1(x) = 1$ for $x \in M_1$, and $\chi_2(x) = 1$ for $x \in M_2$.

By Corollary \ref{order1cor} and Assumption \ref{Extassume} we have a decomposition
\[
\mathcal{N}_\phi = \Upsilon + A
\]
where $\Upsilon$ is an elliptic pseudodifferential operator of order $-1$ and $A$ is an FIO of order $-n/2$ whose canonical relation is a graph.

Since $\Upsilon$ is elliptic it is possible to construct a parametrix $\Upsilon^{-1}$ which is a properly supported pseudodifferential operator of order $1$. This means in particular that
\[
\Upsilon^{-1} \circ \Upsilon = \mathrm{Id} + \mathcal{R}
\]
where $\mathcal{R}$ is an operator with smooth Schwartz kernel. Furthermore, we can construct $\Upsilon^{-1}$ such that its Schwartz kernel is supported sufficiently close to the diagonal so that 
\[
\chi_1^m \circ\Upsilon^{-1} = \chi_1^m \circ \Upsilon^{-1}\circ (\chi_2^m)^2
\]
By the hypotheses, constructions above, and mapping properties of FIOs
\[
\chi_2^m \circ A \circ \chi_1^m: H^s(\Omega_{M}^{1/2}) \rightarrow H^{s+3/2}(\Omega^{1/2}_{M}) \quad \mbox{and} \quad \chi_1^m \circ \Upsilon^{-1} \circ \chi_2^m: H^{s+3/2}(\Omega_{M}^{1/2}) \rightarrow H^{s+1/2}(\Omega_{M}^{1/2})
\]
are compact for any $s$. Putting these together and using Sobolev embedding we find that
\[
\chi_1^m \circ \Upsilon^{-1} \circ (\chi_2^m)^2 \circ A \circ \chi_1^m : H^s(\Omega_{M}^{1/2}) \rightarrow H^s(\Omega_{M}^{1/2}) \quad \mbox{and} \quad \chi_1^m \circ R: H^s(\Omega_{M}^{1/2}) \rightarrow H^s(\Omega_{M}^{1/2})
\]
are compact for any $s$. 

Now suppose that $f \in H^s(\Omega_{M_1}^{1/2})$ has support contained in $M_1$. Then using the construction from the previous paragraph we find that
\be{equation}
f = \chi_1^m \circ \Upsilon^{-1} \circ (\chi_2^m)^2 \circ \mathcal{N}_\phi[f] - \chi_1^m \circ \Upsilon^{-1} \circ (\chi_2^m)^2 \circ A \circ \chi_1^m[f] - \chi_1^m \circ R[f].
\end{equation}
By the previous paragraph the operator $K: H^s(\Omega_{M}^{1/2}) \rightarrow H^s(\Omega_M^{1/2})$ defined by
\[
K = \chi_1^m \circ \Upsilon^{-1} \circ (\chi_2^m)^2 \circ A \circ \chi_1^m - \chi_1^m \circ R
\]
is compact. Using also the mapping property of $\chi_1^m \circ \Upsilon^{-1} \circ (\chi_2^m)^2$ we find that
\be{est1}
\|f \|_{H^s(\Omega^{1/2}_M)} \leq C \left (\| \mathcal{N}_\phi [f]\|_{H^{s+1}(\Omega^{1/2}_M)} + \| K[f] \|_{H^s(\Omega^{1/2}_M)} \right )
\ee
This holds for $f$ with support in $M_1$. For convenience we now introduce the notation
\[
H^s_0(\Omega^{1/2}_{M_1}) = \{ f \in H^s(\Omega^{1/2}_M) \ : \ \mathrm{supp}(f) \subset M_1\}
\]
which is a closed subspace of $H^s(\Omega^{1/2}_M)$. From \eqref{est1} we can establish using standard methods (see for example \cite{StefanovUhlmann04}) that 
\[
\mathrm{ker}(\mathcal{X}_\phi) \cap H^s_0(\Omega^{1/2}_{M_1})
\]
is finite dimensional. Also we can show that for any closed subspace $\mathcal{F}_1$ of $H^s_0(\Omega^{1/2}_{M_1})$ on which $\mathcal{X}_\phi$ is injective
\be{est2}
\|f \|_{H^s(\Omega^{1/2}_M)} \leq C \| \mathcal{N}_\phi [f]\|_{H^{s+1}(\Omega^{1/2}_M)} 
\ee
for all $f \in \mathcal{F}_1$ and any $s \in \mathbb{R}$.

Now suppose that $f \in \mathrm{Ker}(\mathcal{X}_\phi)$ acting on $L^2(\Omega_{M}^{1/2})$. Then since $f |_{\partial_\delta M} = 0$ we have that the support of $f$ is in $M_1$. On the other hand from \eqref{equation} we have
\[
f = -\chi_1^m \circ \Upsilon^{-1} \circ (\chi_2^m)^2 \circ A \circ \chi_1^m [f] - \chi_1^m \circ \mathcal{R}[f].
\]
from which we may conclude that $f \in H^{1/2}(\Omega_{M}^{1/2})$. Repeating this we obtain by a boot strapping argument that $f \in C_c^\infty(\Omega_{M^{int}}^{1/2})$. This proves the first statement of the theorem, and it only remains to prove \eqref{stab}.

Actually, half of \eqref{stab} has already been proven in Lemma \ref{Xcont}, and all that remains is to prove the stability estimate
\be{stabest}
\| f\|_{H^{-1/2}(\Omega_M^{1/2})} \leq C \| \mathcal{X}_\phi[f] \|_{L^2(\Omega_{\partial_- SM}^{1/2})}
\ee
for all $f \in \mathcal{F}$. For this we apply \eqref{est2} with $M$ and $M_1$ replaced by the extensions $\widetilde{M}$ and $\widetilde{M}_1$. Then $\widetilde{\mathcal{F}}_1 = \mathfrak{i}[\mathcal{F}]$ is a closed subspace of $H^{-1/2}_0(\Omega_{\widetilde{M}_1}^{1/2})$ on which $\widetilde{\mathcal{X}}_{\widetilde{\phi}}$ is injective since the extension-by-zero operator $\mathfrak{i}: H^{-1/2}(\Omega_M^{1/2}) \rightarrow H^{-1/2}(\Omega_{\widetilde{M}}^{1/2})$ is an isometric embedding for $s \leq 0$. Using this embedding property again we find that
\[
\| f \|_{H^{-1/2}(\Omega_M^{1/2})} = \| \mathfrak{i}[f] \|_{H^{-1/2}(\Omega_{\widetilde{M}^{1/2}})} \leq C \| \widetilde{\mathcal{N}}_{\widetilde{\phi}} \circ \mathfrak{i}[f]\|_{H^{1/2}(\Omega^{1/2}_{\widetilde{M}})}.
\]
Applying Lemma \ref{Xcont} on $\widetilde{M}$ to this last estimate we have
\[
\| f \|_{H^{-1/2}(\Omega_M^{1/2})} \leq C \| \widetilde{\mathcal{X}}_{\widetilde{\phi}}\circ \mathfrak{i}[f]\|_{L^2(\Omega_{\partial_- S\widetilde{M}}^{1/2})}.
\]
For the final step we use the fact that $R \circ \widetilde{\mathcal{N}}_{\widetilde{\phi}} \circ \mathfrak{i}[f] = \mathcal{N}_\phi[f]$ for $f$ supported in $M$ which we have also used in the proof of Lemma \ref{Ncont} and follows from equation \eqref{explicitN}. Using also the fact that $R$ and $\mathfrak{i}$ are adjoints we have
\[
\begin{split}
\| \widetilde{\mathcal{X}}_{\widetilde{\phi}}\circ \mathfrak{i}[f]\|^2_{L^2(\Omega_{\partial_- S\widetilde{M}}^{1/2})} & =\langle \widetilde{\mathcal{X}}_{\widetilde{\phi}}\circ \mathfrak{i}[f],  \widetilde{\mathcal{X}}_{\widetilde{\phi}}\circ \mathfrak{i}[f] \rangle_{L^2(\Omega_{\partial_- S\widetilde{M}}^{1/2})}  \\
& = \langle f, R \circ \widetilde{\mathcal{N}}_{\widetilde{\phi}} \circ \mathfrak{i}[f]  \rangle_{L^2(\Omega_{\partial_- S M}^{1/2})}\\
& = \langle f, \mathcal{N}_\phi[f]  \rangle_{L^2(\Omega_{\partial_- S M}^{1/2})} \\
& = \langle \mathcal{X}_\phi[f], \mathcal{X}_\phi[f]  \rangle_{L^2(\Omega_{\partial_- S M}^{1/2})} \\
& = \| \mathcal{X}_\phi[f] \|^2_{L^2(\Omega_{\partial_- S M}^{1/2})}.
\end{split}
\]
Thus \eqref{stabest} is proven which completes the proof of Theorem \ref{contthm}.
\end{proof}

\section{Conclusion}

While we have gone some way towards completing the microlocal analysis of the geodesic X-ray transform for nontrapping manifolds, a number of questions remain. Microlocal analysis of the two dimensional case is complete, and indeed in \cite{MoStUh} it is shown, in the two dimensional case, that when there are conjugate points $\mathcal{X}$ can actually cancel singularities. That is, there exist non-smooth distributions $f$ such that $\mathcal{X}[f]$ is smooth. This has the consequence that stable inversion is not possible between any Sobolev spaces, and so the inverse problem of recovering $f$ from $\mathcal{X}[f]$ in this case is severely ill-posed. We have proven that this does not occur in three dimensions or higher provided that Assumptions \ref{basicassume} and \ref{Extassume} are satisfied, and indeed that the problem is only mildly ill-posed in that case. When the assumptions are not satisfied more work is required to determine the degree of ill-posedness.

The failure of Assumption \ref{Extassume} may occur in at least two ways. Firstly, as pointed out in Remark \ref{singremark}, in dimension three and higher the hypothesis of Theorem \ref{mainthm} that $C_S = \emptyset$ fails for generic metrics, although at least in three dimensions this only happens at isolated points which have $D_4$ type singularities. Study of the normal operator $\mathcal{N}_\phi$ near such points is therefore required for a full understanding of the microlocal properties of $\mathcal{X}_\phi$ in three dimenions (which is likely the most interesting case for any application), and in particular understanding of whether the inversion is mildly or severely ill-posed.

Another way the Assumption \ref{Extassume} may fail is the additional requirement that $C_{A_1}$ be a canonical graph. It is at the moment unclear whether this is satisfied generically, if not whether it fails only at isolated points, and what precise impact it might have on the proof of stability estimates as shown in section \ref{stability:sec}. The canonical graph assumption is required for the continuity of $A_1$ between appropriate Sobolev spaces, but weaker versions of such continuity may still hold even when $C_{A_1}$ is not a canonical graph.

As mentioned in the introduction, the geodesic X-ray transform for tensor fields is also of interest and in fact arises naturally in travel time tomography. It is likely the method used in this paper could be extended to the tensor field case with some adjustments, and this is reserved for future work.

\bibliographystyle{abbrv}
\bibliography{Microlocalbib}

\end{document}